\documentclass[reqno, 12pt]{amsart}

\usepackage[utf8]{inputenc}
\usepackage[T1]{fontenc}
\usepackage{microtype}
\usepackage{lmodern}

\usepackage{
    graphicx, 
    fancyhdr, 
    amscd,
    amssymb,
    amsmath, 
    tikz, 
    tikz-cd, 
    mathtools, 
    transparent,
    xfrac,
}
\usetikzlibrary{shapes.geometric, tqft}

\usepackage[shortlabels]{enumitem}

\usepackage[%
    colorlinks=true,
    linkcolor=blue,
]{hyperref}

\usepackage[%
    capitalise, 
    nameinlink,
    noabbrev,
]{cleveref}

\usepackage[
    backend=biber,
    style=numeric-comp,
]{biblatex}
\AtBeginBibliography{\small}

\usepackage{overpic}

\usepackage{caption}
\usepackage[left=3.25cm,right=3.25cm,top=3cm,bottom=3cm,head=.25cm,foot=1.5cm]{geometry}

\numberwithin{equation}{section}

\newtheorem{theorem}{Theorem}[section]
\newtheorem{prop}[theorem]{Proposition}
\newtheorem{lemma}[theorem]{Lemma}
\newtheorem{corollary}[theorem]{Corollary}

\theoremstyle{definition}

\newtheorem{example}[theorem]{Example}
\newtheorem{remark}[theorem]{Remark}

\def\ZZ{\mathbb{Z}}
\def\RR{\mathbb{R}}

\def\pr{\mathrm{pr}}
\def\Hom{\mathrm{Hom}}
\def\SO{\mathrm{SO}}
\def\O{\mathrm{O}}
\def\Tor{\mathrm{Tor}}
\def\Ext{\mathrm{Ext}}
\def\Pin{\mathrm{Pin}}
\def\PinStr{\mathcal{P}\mathrm{in}}
\def\Split{\mathcal{S}\mathrm{plit}}
\def\Spin{\mathrm{Spin}}
\def\Fr{\mathbb{F}}
\def\Or{\mathbb{L}}
\def\Pinm{\Pin^-}
\def\PinBor1{\Omega_1^{\Pinm}}
\def\normb{\nu}
\def\sq{\mathrm{Sq}}
\def\coker{\mathrm{coker}}
\def\susp{\mathrm{S}}
\def\RP{\mathbb{R}\mathrm{P}}

\def\framing{\varphi}

\newcommand\restr[2]{\ensuremath{\left.#1\right|_{#2}}}
\newcommand\smod[1]{\ensuremath{\mathrm{mod} \: #1}}

\definecolor{myblue}{HTML}{0089CF}
\definecolor{myorange}{HTML}{CF4600}

\begin{document}
	
	\title{A geometric computation of cohomotopy groups in co-degree one}
	
	\author{Michael Jung, Thomas O. Rot}
	\address{Department of Mathematics, Vrije Universiteit Amsterdam}
	\email{m.jung@vu.nl, t.o.rot@vu.nl}
    \keywords{cohomotopy, Pontryagin-Thom construction, pin structures, homology with local coefficients, vector bundles, Euler class}
    \subjclass{55Q55, 57R22, 55N25, 57R15}
	
	\begin{abstract}
        Using geometric arguments, we compute the group of homotopy classes of maps from a closed $(n+1)$-dimensional manifold to the $n$-sphere for $n \geq 3$.
        Our work extends results from Kirby, Melvin and Teichner for closed oriented 4-manifolds and from Konstantis for closed $(n+1)$-dimensional spin manifolds, considering possibly non-orientable and non-spinnable manifolds.
        In the process, we introduce two types of manifolds that generalize the notion of odd and even 4-manifolds.
        Furthermore, for the case that $n \geq 4$, we discuss applications for rank~$n$ spin vector bundles and obtain a refinement of the Euler class in the cohomotopy group that fully obstructs the existence of a non-vanishing section.
	\end{abstract}
	
	\maketitle
	\thispagestyle{empty}

    \section{Introduction}

    The cohomotopy sets $\pi^n(X) = [X^{n+k}, S^n]$ of (unpointed) homotopy classes of maps from an $(n+k)$-dimensional cell complex $X$ into an $n$-sphere play an important role in algebraic and differential topology.
    If $n$ is large enough, this set can be given a group structure (see Spanier~\cite{bib:spanier:cohomotopy_groups}).
    In the case when $k=1$ and $n \geq 3$, Steenrod's main theorem~\cite[Theorem 28.1, p.\ 318]{bib:steenrod:cocycles} implies that $\pi^n(X)$ fits into the short exact sequence of abelian groups
    \begin{equation}\label{eq:steenrod:ses}
        \begin{tikzcd}[sep=small]
            0 \arrow[r] & \sfrac{H^{n+1}(X; \ZZ_2)}{(\sq^2 \circ r)\!\left(H^{n-1}(X; \ZZ)\right)} \arrow[r] & \pi^n(X) \arrow[r] & H^n(X; \ZZ) \arrow[r] & 0,
        \end{tikzcd}
    \end{equation}
    where $r \colon H^{n-1}(X; \ZZ) \to H^{n-1}(X; \ZZ_2)$ is the $\smod 2$ reduction, and the epimorphism pulls back the generator of $H^n(S^n ; \ZZ)$.
    The question about the corresponding group extension has been revisited by Taylor~\cite{bib:taylor:principal} using methods of Larmore and Thomas~\cite{bib:lam_tho:extension}.
    While this approach relies on homotopical techniques involving Postnikov towers, it is worth exploring a more geometric approach if $X$ is a manifold.
    Notably, Kirby, Melvin, and Teichner~\cite{bib:kir_mel_tei:coh_4manifolds} provided a geometric proof for the case of closed, oriented 4-manifolds, contributing to a deeper understanding of the subject.    
    Additionally, Konstantis~\cite{bib:konstantis:counting_inv} presented a geometric proof for closed $(n+1)$-dimensional spin manifolds, introducing a new counting invariant in the process.
    The key tool in both cases is the Pontryagin-Thom construction that provides an isomorphism $\pi^n(X^{n+k}) \cong \Fr_k(X)$, where $\Fr_k(X)$ denotes the set of normally framed $k$-dimensional submanifolds of $X$ up to normally framed bordism in $X \times [0,1]$.

    In this article, we aim to complete the geometric picture when $k=1$ and $n \geq 3$.
    More precisely, we provide a geometric computation of $\pi^n(X)$ for closed $(n+1)$-dimensional smooth manifolds that are not necessarily spin or orientable.
    We find that $\Fr_1(X)$ fits into a dual version of~\eqref{eq:steenrod:ses}:
    \begin{equation}\label{eq:ses:framed}
        \begin{tikzcd}[sep=small]
            0 \arrow[r] & \ker(h) \arrow[r] & \Fr_1(X) \arrow[r, "h"] & H_1(X; o_X) \arrow[r] & 0,
        \end{tikzcd}
    \end{equation}
    where $h$ maps the bordism class of a normally framed link to the corresponding twisted fundamental class in homology with coefficients in the orientation sheaf~$o_X$. 
    Morally speaking, the map $h \colon \Fr_1(X) \to H_1(X; o_X)$ forgets the normal framing but remembers its orientation.
    This allows us to use geometric techniques to determine $\ker(h)$ and compute the extension of the sequence.

    As it turns out, the result depends on a certain classification of manifolds that generalizes the notion of odd and even oriented 4-manifolds.
    More precisely, if $n=3$, that is if $X$ is a 4-manifold, we say that $X$ is of \emph{type~I} if there is a closed embedded surface with an orientable normal bundle but non-vanishing $\smod 2$ self-intersection.
    For $n \geq 4$, we say that $X$ is of \emph{type~I} if there is a closed embedded surface whose normal bundle is orientable but not trivial.
    The manifold $X$ is of \emph{type~II} if it is \underbar{not} of type~I.
    We now state the main theorem of this article.
    \begin{theorem}\label{thm:main}
        If $X$ is of type~I, the forgetful map $h \colon \Fr_1(X) \to H_1(X; o_X)$ is an isomorphism.
        If $X$ is of type~II, then there is a short exact sequence of abelian groups
        \begin{equation}\label{eq:ses:typeii}
            \begin{tikzcd}[sep=small]
                0 \arrow[r] & \ZZ_2 \arrow[r] & \Fr_1(X) \arrow[r, "h"] & H_1(X; o_X) \arrow[r] & 0
            \end{tikzcd}
        \end{equation}
        whose extension is classified by the unique element in $\Ext(H_1(X; o_X), \ZZ_2)$ that is mapped to $w^2_1(X) + w_2(X)$ under $\delta$ in the universal coefficient sequence
        \begin{equation*}
            \begin{tikzcd}[sep=small]
                0 \arrow[r] & \Ext(H_{1}(X;o_X), \ZZ_2) \arrow[r, "\delta"] & H^2(X; \ZZ_2) \arrow[r] & \Hom(H_2(X; o_X), \ZZ_2) \arrow[r] & 0 \mathrlap{.}
            \end{tikzcd}
        \end{equation*}
    \end{theorem}
    Recall that $w^2_1(X) + w_2(X)$ is the primary obstruction of finding a $\Pinm$-structure on $X$.
    It follows from the main theorem that the sequence~\eqref{eq:ses:typeii} splits if and only if $X$ admits a $\Pinm$-structure.
    In that case, we also obtain a correspondence between splitting maps and $\Pinm$-structures, which is stated in the following theorem:
    \begin{theorem}\label{thm:splitting_pin}
        Suppose $X$ is $\Pinm$.
        Then the sequence~\eqref{eq:ses:typeii} splits.
        Moreover, if $X$ is orientable, all equivalent $\Spin$-structures of $X$ are in 1-to-1 correspondence with splitting maps of~\eqref{eq:ses:typeii}.
        On the other hand, if $X$ is not orientable, all equivalent $\Pinm$-structures of $X$ are in 2-to-1 correspondence with splitting maps of~\eqref{eq:ses:typeii}.
        In that case, the difference is given by the action of $w_1(X)$.
    \end{theorem}
    
    At last, we discuss applications for vector bundles and generalize the invariant introduced in~\cite{bib:konstantis:counting_inv}.
    More precisely, for the case that $n \geq 4$, we construct a refinement of the Euler class in the cohomotopy group $\pi^n(X)$ that gives a full obstruction of finding a non-vanishing section for spin vector bundles over $X$ of rank~$n$.
    If $X$ is of type~I, we obtain an explicit result:
    \begin{theorem}\label{thm:vb_typei}
        Suppose $X$ is an $(n+1)$-dimensional connected closed manifold of type I with $n \geq 4$.
        Let $E \to X$ be an oriented spin vector bundle of rank $n$.
        Then the Euler class $e(E)$ is zero if and only if $E$ admits a non-vanishing section.
    \end{theorem}
    If $X$ has a $\Pinm$-structure, we can apply the associated splitting map $\kappa \colon \Fr_1(X) \to \ZZ_2$ to the newly defined obstruction class.
    The corresponding number $\kappa(E)$ is called the \emph{degree of $E$} and well-defined if the Euler class $e(E)$ vanishes.
    This leads to an explicit result for vector bundles over $\Pinm$ manifolds:
    \begin{theorem}\label{thm:vb_typeiia}
        Suppose $X$ is an $(n+1)$-dimensional connected closed manifold that is $\Pinm$ where $n \geq 4$.
        Let $E \to X$ be an oriented spin vector bundle of rank $n$.
        Then $E$ admits a non-vanishing section if and only if both the Euler class $e(E)$ and the degree $\kappa(E)$ vanish.
    \end{theorem}

    The article is organized as follows.
    We start by revisiting fundamental definitions and classical computations in \cref{sec:preliminaries}.
    Moving on to \cref{sec:kernel}, we present a geometric interpretation of the kernel $\ker(h)$.
    In \cref{sec:typei}, we prove the first part of \cref{thm:main} concerning manifolds of type~I.
    Shifting focus to type~II manifolds, in \cref{sec:typeii} we discuss the proof of the second part of \cref{thm:main}.
    Dedicating \cref{sec:typeiia} to $\Pinm$ manifolds, we show the interplay between splitting maps and $\Pinm$-structures, and we derive \cref{thm:splitting_pin}.
    Finally, in \cref{sec:vector_bundles}, we explore the implications for vector bundles and validate \cref{thm:vb_typei} as well as \cref{thm:vb_typeiia} both as a corollary of \cref{thm:nonvanishing_section}.
  
    \subsection*{Acknowledgements}
    The authors would like to express their gratitude to Panagiotis Konstantis for his valuable time and insightful discussions regarding his results.
    They also extend their appreciation to Stefan Behrens and Lauran Toussaint for engaging in fruitful discussions that contributed to the project.
    In closing, our thanks go to the referee for their careful reading and insightful comments, which greatly enhanced the quality of this article.

    \section{Preliminaries}
    \label{sec:preliminaries}
    In what follows we assume that $X$ is a closed, connected $(n+1)$-dimensional smooth manifold, where $n \geq 3$.
    If $\iota \colon M \looparrowright X$ is an immersion, we denote by $\normb_\iota$ the associated normal bundle.
    Provided that $\iota \colon M \hookrightarrow X$ is an embedding, and if the context allows, we may suppress the map $\iota$ and denote the normal bundle simply by~$\normb_M$.
    A trivialization of a vector bundle up to a homotopy through trivializations is referred to as a \emph{framing}.

    \subsection*{Pin-structures}

    A $\Pin$-structure generalizes the notion of a $\Spin$-structure to non-orientable vector bundles.
    While $\SO(r)$ is double covered by $\Spin(r)$, the group $\O(r)$ has two possible central $\ZZ_2$ extensions, denoted by $\Pin^{\pm}(r)$.
    Suppose $E \to X$ is a vector bundle over $X$ of rank $r$.
    A $\Pin^{\pm}$-principal bundle over $X$ that lifts the frame bundle of $E$ is called a \emph{$\Pin^\pm$-structure over $E$}.
    In our article, we focus entirely on $\Pin^-$-structures.
    The obstruction to putting a $\Pinm$-structure on $E$ is given by $w_1^2(E) + w_2(E)$.
    If $E$ admits a $\Pinm$-structure, the set of all equivalent $\Pinm$-structures is a $H^1(X; \ZZ_2)$-torsor.
    In other words, the group $H^1(X; \ZZ_2)$ acts simply transitive on the set of equivalent $\Pinm$-structures.
    We denote the (stable) bordism group of $r$-dimensional manifolds with an $\Pinm$-structure on the \emph{tangent bundle} by $\Omega^{\Pinm}_r$.
    An in-depth analysis of both $\Pin^-$ and $\Pin^+$-structures and their low-dimensional bordism groups is done by Kirby and Taylor~\cite{bib:kir_tay:pin}.

    \subsection*{Homology with twisted coefficients}

    Since we are interested in non-orientable manifolds as well, we need to consider \emph{homology with local coefficients}.
    An exposition of their definition and properties, also for cohomology, can be found in Spanier~\cite{bib:spanier:alg_top} as well as Davis and Kirk~\cite{bib:dav_krik:alg_top}.
    Recall that the \emph{orientation sheaf} is defined as the local coefficient system
    \begin{align*}
        o_X = \bigsqcup_{x \in X} H_{n+1}(X, X- \{x\}; \ZZ),
    \end{align*}
    sometimes referred to as \emph{twisted coefficients}.
    Our object of interest is the associated homology $H_*(X; o_X)$.

    We are particularly interested in the following functorial property.
    If $\iota \colon M \looparrowright X$ is an immersion of an $(k+1)$-dimensional manifold, we get an isomorphism of orientation bundles $$\iota^* \left( \bigwedge\nolimits^{n+1} TX \right) \cong \bigwedge\nolimits^{k+1} TM \otimes \bigwedge\nolimits^{n-k} \normb_\iota.$$
    An orientation of $\normb_\iota$ induces an isomorphism $\bigwedge\nolimits^{k+1} TM \cong \iota^* ( \bigwedge\nolimits^{n+1} TX )$, and hence an isomorphism of local coefficients $o_M \cong \iota^*o_X$.
    This then defines a pushforward $\iota_* \colon H_*(M; o_M) \to H_*(X; o_X)$.

    Next, we discuss a universal coefficient theorem for local coefficients.%
    \footnote{
        Spanier mentions the possibility of deducing such a theorem in \cite[p.~283]{bib:spanier:alg_top}, but leaves the statement entirely open as an exercise.
    }
    Suppose $\Gamma$ is a local coefficient system over $X$, and $G$ is an abelian group.
    We obtain a new local coefficient system
    \begin{align*}
        \Hom(\Gamma, G) := \bigsqcup_{x \in X} \Hom(\Gamma(x), G).
    \end{align*}
    A possible formulation of a universal coefficient theorem for local coefficents in cohomology is then as follows:
    \begin{theorem}\label{thm:twisted_uct}
        Let $\Gamma$ be a local system of $X$ whose fibers are free abelian groups.
        Then we have a natural short exact sequence
        \[
            \begin{tikzcd}[sep=tiny]
                0 \arrow[r] & \Ext(H_{k-1}(X;\Gamma), G) \arrow[r] & H^k(X; \Hom(\Gamma, G)) \arrow[r] & \Hom(H_k(X; \Gamma), G) \arrow[r] & 0.
            \end{tikzcd}
        \]
        This sequence splits, but not naturally.
    \end{theorem}
    \begin{proof}
        Observe that there is a canonical chain isomorphism
        \begin{align*}
            C^*(X; \Hom(\Gamma, G)) \xrightarrow{\cong} \Hom(C_*(X; \Gamma), G),
        \end{align*}
        where the boundary maps on the right are given by precomposition with $\partial$.
        As the chain complex $C_*(X; \Gamma)$ is free abelian when $\Gamma$ has free abelian fibers, the proof remains the same as for ordinary cohomology.
        We refer to Hatcher~\cite[Theorem~3.2]{bib:hatcher:alg_top} for the details.
    \end{proof}
    Our special case of interest entails $\Gamma = o_X$ and $G = \ZZ_2$.
    Then $\Hom(o_{X,x}, \ZZ_2) \cong \ZZ_2$ for any $x \in X$, and therefore $\Hom(o_X, \ZZ_2) \cong X \times \ZZ_2$.
    For that special case, we obtain the sequence
    \begin{equation}\label{eq:twisted_uct}
        \begin{tikzcd}[sep=small]
            0 \arrow[r] & \Ext(H_{k-1}(X;o_X), \ZZ_2) \arrow[r, "\delta"] & H^k(X; \ZZ_2) \arrow[r, "r^*"] & \Hom(H_k(X; o_X), \ZZ_2) \arrow[r] & 0,
        \end{tikzcd}
    \end{equation}
    where $r \colon H_k(X; o_X) \to H_k(X; \ZZ_2)$ is the $\smod 2$ reduction with $H^k(X; \ZZ_2)$ and $\Hom(H_k(X; \ZZ_2), \ZZ_2)$ identified.

    Analogous to the (untwisted) oriented case, we have \emph{twisted Poincar\'e duality} for any manifold:
    \begin{align*}
        \mathrm{PD}_{\mathrm{tw}} \colon H^k(X; \ZZ) \xrightarrow{\cong} H_{n+1-k}(X; o_X).
    \end{align*}
    To prove this, replace the ordinary cap product
    \begin{align*}
        \cap\colon C_p(X; \ZZ) \times C^q(X; \ZZ) \to C_{p-q}(X;\ZZ)
    \end{align*}
    with the twisted cap product
    \begin{align*}
        \cap\colon C_p(X; o_X) \times C^q(X; \ZZ) \to C_{p-q}(X; o_X).
    \end{align*}
    Every manifold has a preferred \emph{twisted fundamental class} because $$H_{n+1}(X, X - \{x\}; o_X) \cong H_{n+1}(X, X - \{x\}; \ZZ) \otimes H_{n+1}(X, X - \{x\}; \ZZ)$$ has a preferred generator.
    The isomorphism can be understood through a twisted version of the universal coefficient theorem in homology, but we omit the details.
    The Poincar\'e duality isomorphism is then given by the twisted cap product with the twisted fundamental class.

    \subsection*{Computations by Thom}

    It follows from Thom's seminal work~\cite{bib:thom:seminal_work} that every homology class in degree $d$ with integer coefficients can be represented by a closed $d$-dimensional submanifold as long as $d$ remains sufficiently small.
    The twisted case is analogous and is explained by Atiyah~\cite{bib:atiyah:bordism} in more detail.
    The idea is to consider the commutative diagram
    \[
        \begin{tikzcd}
            \left[X, \mathrm{M}\SO(k)\right] \arrow[r, "\cong"] \arrow[d, "u_*"'] & \Or_{n+1-k}(X; o_X) \arrow[d] \\
            H^{k}(X; \ZZ) \arrow[r, "\cong", "\mathrm{PD}_{\mathrm{tw}}"'] & H_{n+1-k}(X; o_X).
        \end{tikzcd}    
    \]
    Here, we denote by $\Or_{d}(X; o_X)$ the set of closed $d$-dimensional submanifolds of $X$ with an orientation of the normal bundle up to bordism in $X \times [0,1]$ with an orientation of the normal bundle.
    The top horizontal arrow is given by the Pontryagin-Thom construction.
    The left vertical arrow is induced by the Thom class ${u \colon \mathrm{M}\SO(k) \to K(\ZZ, k)}$, where $\mathrm{M}\SO(k)$ denotes the Thom space of the universal bundle for $\SO(k)$.
    The right vertical arrow is given by the pushforward of the twisted fundamental class.
    Thom~\cite{bib:thom:seminal_work} has shown that $u \colon \mathrm{M}\SO(k) \to K(\ZZ, k)$ is $(k+2)$-connected.
    That is if we assume $k \geq n-1$, the left vertical arrow is an isomorphism, and we obtain $H_d(X; o_X) \cong \Or_{d}(X; o_X)$ for $d=1,2$.
    Similar results hold for $\ZZ_2$ coefficients.

    \subsection*{Two types of manifolds}

    Throughout this article, we consider two types of manifolds.
    Denote by $(w^2_1 + w_2) \colon H_2(X; \ZZ_2) \to \ZZ_2$ the functional associated with $w_1^2(X) + w_2(X)$ and by $r \colon H_2(X; o_X) \to H_2(X; \ZZ_2)$ the $\smod 2$ reduction.
    We say that
    \begin{align*}
        X ~ \text{is of type I} \quad &:\Longleftrightarrow \quad (w_1^2 + w_2) \circ r \not\equiv 0, \\
        X ~ \text{is of type II} \quad &:\Longleftrightarrow \quad (w_1^2 + w_2) \circ r \equiv 0.
    \end{align*}
    It is clear from the definition that $\Pinm$ manifolds are automatically of type II.
    But not all manifolds of type II are $\Pinm$.
    We say that $X$ is of \emph{type IIa} if it is $\Pinm$ and of \emph{type IIb} if it is of type II but not $\Pinm$.

    We now give a geometric interpretation of this condition in terms of normal bundles of surfaces.
    \begin{lemma}\label{lem:func_measures_w2}
        Let $\iota \colon \Sigma \looparrowright X$ be an immersed closed surface.
        Then $${(w_1^2 + w_2)(\iota_*[\Sigma]_2) = \left\langle w_2(\normb_\iota), [\Sigma]_2 \right\rangle}.$$
    \end{lemma}
    \begin{proof}
        A simple computation reveals
        \begin{align*}
            (w_1^2 + w_2)(\iota_*[\Sigma]_2) &= \left\langle w_1^2(X) + w_2(X), \iota_*[\Sigma]_2 \right\rangle \\
                    &= \left\langle \iota^*\!\left(w_1^2(X) + w_2(X)\right), [\Sigma]_2 \right\rangle \\
                    &= \left\langle w_1^2(\Sigma) + w_2(\Sigma) + w_2(\normb_\iota) + w_1(\Sigma) w_1(\normb_\iota) + w^2_1(\normb_\iota), [\Sigma]_2 \right\rangle \\
                    &= \left\langle w_2(\normb_\iota) + w_1(\Sigma) w_1(\normb_\iota) + w^2_1(\normb_\iota), [\Sigma]_2 \right\rangle \\
                    &= \left\langle w_2(\normb_\iota), [\Sigma]_2 \right\rangle.
        \end{align*}
        Here we used the fact that $\iota^*(TX) \cong T \Sigma \oplus \normb_\iota$, and that all surfaces are $\Pinm$, which means $w_1^2(\Sigma) + w_2(\Sigma) = 0$.
        The last step follows from Wu's formula which implies $w_1(\Sigma) w_1(\normb_\iota) = \sq^1(w_1(\normb_\iota))$ and hence $w_1(\Sigma) w_1(\normb_\iota) = w^2_1(\normb_\iota)$.
    \end{proof}

    \begin{prop}\label{prop:geom_typei}
        If there exist an immersed closed surface $\iota \colon \Sigma \looparrowright X$ with $w_1(\normb_\iota) = 0$ and $w_2(\normb_\iota) \neq 0$, then $X$ is of type~I.
        Conversely, if $X$ is of type I, then there exists an embedded closed surface $\iota \colon \Sigma \hookrightarrow X$ such that $w_1(\normb_\iota) = 0$ and $w_2(\normb_\iota) \neq 0$.
    \end{prop}
    \begin{proof}
        Suppose $\iota \colon \Sigma \looparrowright X$ is an immersed closed surface with $w_1(\normb_\iota)=0$.
        We have seen in \cref{lem:func_measures_w2} that the functional $w_1^2+w_2$ measures the triviality of the second Stiefel-Whitney class of the normal bundle $w_2(\normb_\iota)$.
        Together with functoriality in twisted coefficients, this shows the first implication.
        The converse follows from the fact that $H_2(X;o_X)$ is generated by embedded closed surfaces with an orientation of the normal bundle as discussed in the previous paragraph.
    \end{proof}
    \begin{remark}
        It is worth mentioning that \cref{prop:geom_typei} implies that if there exists an immersed closed surface $\iota \colon \Sigma \looparrowright X$ with $w_1(\normb_\iota) = 0$ and $w_2(\normb_\iota) \neq 0$, then there automatically exists an embedded one with that property.
    \end{remark}
    Equivalently, we obtain the geometric statement for manifolds of type~II.
    \begin{corollary}\label{cor:geom_typeii}
        If all embedded surfaces $\iota \colon \Sigma \hookrightarrow X$ with $w_1(\normb_\iota) = 0$ have $w_2(\normb_\iota) = 0$, then $X$ is of type~II.
        Conversely, if $X$ is of type~II, then all immersed surfaces $\iota \colon \Sigma \looparrowright X$ with $w_1(\normb_\iota) = 0$ must have $w_2(\normb_\iota) = 0$.
    \end{corollary}
    When $X$ is a 4-manifold, the class $w_2(\normb_\Sigma)$ reports $\smod 2$ self-intersection for normally oriented surfaces $\Sigma \subset X$.
    If $X$ is in addition oriented, type~I is equivalent to $X$ being \emph{odd}, and type~II is equivalent to $X$ being \emph{even}.
    For higher dimensions, that is if $n \geq 4$, we can give another geometric description.
    The key ingredient is the following remark.
    \begin{remark}\label{rmk:vb_trivial_sw_class}
        Recall that any vector bundle $E$ of rank $n \geq 3$ over a 2-dimensional cell complex $\Sigma$ is trivializable if and only if $w_1(E)$ and $w_2(E)$ vanish.
        This can be seen as follows.
        Since $n \geq 3$, the second Stiefel Whitney class $w_2(E)$ provides a full obstruction of finding an $(n-1)$-frame over~$\Sigma$, see Milnor and Stasheff~\cite[§13]{bib:mil_stash:char_classes}.
        Then the $(n-1)$-frame can be completed to an $n$-frame over $\Sigma$ if and only if the bundle is orientable, or equivalently if and only if $w_1(E)$ vanishes.
    \end{remark}
    With that remark, we deduce the following corollary:
    \begin{corollary}
        Let $n \geq 4$.
        Then $X$ is of type I if and only if there is a closed surface $\Sigma \subset X$ such that $\normb_\Sigma$ is orientable but not trivializable.
        Equivalently, $X$ is of type II if and only if all closed surfaces $\Sigma \subset X$ for which $\normb_\Sigma$ is orientable have that $\normb_\Sigma$ is trivializable.
    \end{corollary}
    
    \subsection*{The forgetful map}

    For any $k$, there is an obvious map ${\Fr_k(X) \to \Or_k(X; o_X)}$ that forgets the framing but remembers the orientation of the normal bundle.
    For $k=1$, we can identify $H_1(X; o_X) \cong \Or_{1}(X; o_X)$, and hence we obtain a group homomorphism $$h \colon \Fr_1(X) \to H_1(X; o_X).$$
    Since any orientable vector bundle over a 1-dimensional cell complex is trivializable, this map is surjective.
    Therefore, $\Fr_1(X)$ always fits into the short exact sequence~\eqref{eq:ses:framed}.
    In the following sections, we determine $\ker(h)$ and classify the extension with geometric methods.

    \section{The kernel of the forgetful map}
    \label{sec:kernel}

    To gain insights into the kernel of $h \colon \Fr_1(X) \to H_1(X; o_X)$, we need to familiarize ourselves with certain properties of framings over surfaces with boundary.
    For what follows, we call a framing \emph{positive} if it is compatible with a given orientation.
    \begin{lemma}\label{lem:framings_first_cohomology}
        Let $Y$ be a 2-dimensional cell complex and $E \to Y$ be a trivializable vector bundle of rank $n \geq 3$ with a given orientation.
        Then distinct positive framings of $E$ are in one-to-one correspondence with elements of $H^1(Y; \ZZ_2)$.
    \end{lemma}
    \begin{proof}
        All distinct positive framings of $E$ are in one-to-one correspondence with homotopy classes in $[Y, \SO(n)]$ via composition with a fixed positive framing.
        We represent the generator of $H^1(\SO(n), \ZZ_2) \cong \ZZ_2$ by a map $u \colon \SO(n) \to K(\ZZ_2, 1)$.
        Since ${\pi_2(\SO(n))=0}$, the map $u$ is 3-connected.
        Hence, the induced map $u_* \colon [Y, \SO(n)] \to H^1(Y; \ZZ_2)$ is an isomorphism.
    \end{proof}
	\begin{prop}\label{prop:unique_framing}
	    Let $\Sigma$ be a compact connected surface with boundary components $\partial\Sigma = C \sqcup C_1 \sqcup \dots \sqcup C_k$, where $k \geq 0$.
	    Let $E$ be an oriented vector bundle of rank $n \geq 3$ over~$\Sigma$.
	    Then any choice of positive framing of $E$ over $C_1, \dots, C_k$ extends to a positive framing over~$\Sigma$.
        Moreover, the induced framing over $C$ of any such extension only depends on the choice made over $C_1, \dots, C_k$.
	\end{prop}
    \begin{proof}
        Since $E$ is oriented and $\Sigma$ is homotopy equivalent to a 1-dimensional cell complex, we may fix a positive framing of $E$.
        Then by \cref{lem:framings_first_cohomology}, positive framings over any subcomplex $\iota \colon Y \hookrightarrow \Sigma$ are in one-to-one correspondence with elements in $H^1(Y; \ZZ_2)$.
        The restriction of a framing over $\Sigma$ to a subcomplex $Y$ is then given by the induced map $\iota^* \colon H^1(\Sigma; \ZZ_2) \to H^1(Y; \ZZ_2)$.
        If we set $Y = C_1 \sqcup \dots \sqcup C_k$, then we have the following short exact sequence in relative cohomology:
        \[
        \begin{tikzcd}
            H^1(\Sigma; \ZZ_2) \arrow[r, "\iota^*"] & H^1(Y; \ZZ_2) \arrow[r] & {H^2(\Sigma, Y ; \ZZ_2) \cong 0}\mathrlap{.}
        \end{tikzcd}
        \]
        This sequence shows that all positive framings over $Y$ can be extended across $\Sigma$ and further restricted to $C$ via $\jmath^* \colon H^1(\Sigma;\ZZ_2) \to H^1(C; \ZZ_2)$, where $\jmath \colon C \hookrightarrow \Sigma$ is the inclusion of $C$.
        For uniqueness, suppose $\alpha \in H^1(\Sigma; \ZZ_2) \cong \Hom(H_1(\Sigma; \ZZ_2), \ZZ_2)$.
        Then the restriction $\jmath^*\alpha$ is determined by
        \begin{align*}
            (\jmath^* \alpha)([C]_2) &= \alpha(\jmath_*[C]_2) \\
                &= \alpha(\iota_*[Y]_2) \\
                &= \iota^*\alpha([Y]_2),
        \end{align*}
        and hence only determined by the restriction $\iota^*\alpha$.
        The second equality follows from $\jmath_*[C]_2 = \iota_* [Y]_2 \in H_1(\Sigma; \ZZ_2)$ since $\Sigma$ is a bordism between $C$ and $Y = C_1 \sqcup \dots \sqcup C_k$.
    \end{proof}
    The case $k=0$ deserves special attention.
    \begin{corollary}\label{cor:unique_framing_one_boundary}
        Under the assumptions of \cref{prop:unique_framing}, if $\Sigma$ has only one boundary component $C$, then there is exactly one framing over $C$ that extends over the whole surface~$\Sigma$.
    \end{corollary}
    We wish to isotope bordisms in $X \times [0,1]$ to cut off disks whose boundaries respect slices.
    If the projection to $[0,1]$ is Morse, this is immediate, and the subsequent proposition explains to us how we can achieve that.
	\begin{lemma}\label{lem:isotope_bordism}
		Let $\Sigma$ be a compact surface and $\iota_0 \colon \Sigma \hookrightarrow X \times [0,1]$ be an embedding mapping boundaries to boundaries.
        Then there is another embedding $\iota_1 \colon \Sigma \to X \times [0,1]$ mapping boundaries to boundaries such that
		\begin{enumerate}[i)]
			\item $\iota_0$ and $\iota_1$ are isotopic relative to $\partial \Sigma$,
			\item $(\pr_2 \circ \iota_1) \colon \Sigma \to [0,1]$ is a Morse function.
		\end{enumerate}
	\end{lemma}
	\begin{proof}
        Let $f \colon \Sigma \to [0,1]$ be a Morse function on $\Sigma$.
        If $\Sigma$ has boundary, we may choose the Morse function $f$ in such a way that $\restr{f}{\partial \Sigma} = \restr{\left(\pr_2 \circ \iota_0\right)}{\partial\Sigma}$.
	    Consider the homotopy ${\iota_t \colon \Sigma \to X \times [0,1]}$ with
        \begin{align*}
            \iota_t = \big(\left(\mathrm{pr}_1 \circ \iota_0\right), \, \left(\mathrm{pr}_2 \circ \iota_0\right) + t\left[ f - \left(\mathrm{pr}_2 \circ \iota_0 \right) \right] \big).
        \end{align*}
		We may choose $f$ sufficiently close to $(\mathrm{pr}_2 \circ \iota_0)$ such that $\iota_t$ becomes an isotopy.
        This is possible because embeddings are open in the strong topology.
	\end{proof}
    For what follows, we fix a contractible circle $U \subset X$ and a disk $D\subset X \times [0,1]$ that bounds $U$.
    Then $U$ has two possible normal framings.
    One framing extends over $D$, and we denote the corresponding framed circle by $U_0$.
    The other framing does not extend over $D$, and the associated framed circle is denoted by $U_1$.
    \begin{lemma}\label{lem:ker_h_generator}
        The subgroup $\ker(h) \subset \Fr_1(X)$ is generated by $[U_1]$.
    \end{lemma}
    \begin{proof}
        For any element $[L, \framing] \in \ker(h)$, there is a normally oriented bordism $\Sigma \hookrightarrow X \times [0,1]$ from $L$ to the empty set.
        Notice that $\Sigma$ may not be connected, i.e.\ $\Sigma = \Sigma_1 \sqcup \dots \sqcup \Sigma_k$.
        We isotope each connected component $\Sigma_i$ according to \cref{lem:isotope_bordism}.
        Then we can cut off an embedded disk $D_i$ from each $\Sigma_i$ such that $\partial D_i$ sits in one slice of $X \times [0,1]$.
        Each $\partial D_i$ carries a uniquely determined normal framing via \cref{prop:unique_framing}.
        We thus get an induced normal framing of the disjoint union $\partial D_1 \sqcup \dots \sqcup \partial D_k$ which makes it framed bordant to $(L, \framing)$.
        Now, each framed circle $\partial D_i$ is framed bordant to $U_0$ or $U_1$.
        But since $[U_0] = 0$ in $\Fr_1(X)$, the element $[L, \framing]$ is a multiple of $[U_1]$.
    \end{proof}
    Since two copies of $U_1$ are framed bordant to $U_0$ through a pair of pants, the element $[U_1]$ is of order at most two, and we get as a proposition:
    \begin{prop}\label{prop:u1_order_two}
        The forgetful map $h \colon \Fr_1(X) \to H_1(X; o_X)$ is either an isomorphism or a 2-to-1 epimorphism.
    \end{prop}

    \section{Manifolds of type I}
    \label{sec:typei}

    We have seen that $[U_1]$ generates $\ker(h)$ by \cref{lem:ker_h_generator}.
    This allows us to restate the first part of \cref{thm:main} for type~I manifolds as follows:
    \begin{theorem}
        The manifold $X$ is of type I if and only if $[U_1] = 0$ in $\Fr_1(X)$.
    \end{theorem}
    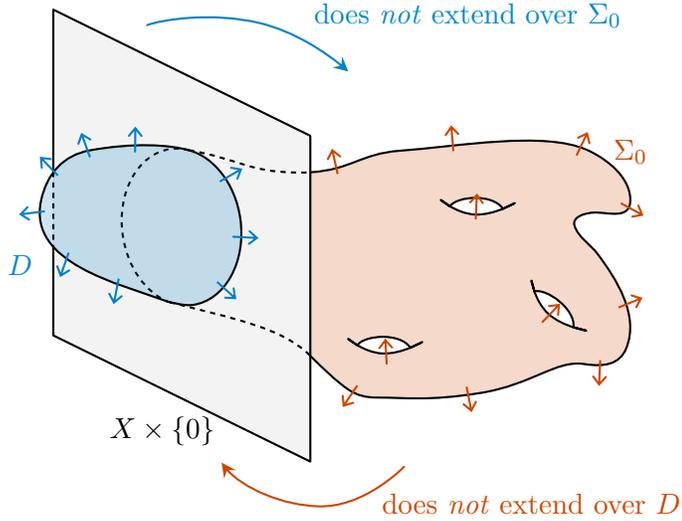
\begin{figure}[t]
        \centering
        \begin{overpic}[width=.6\textwidth]{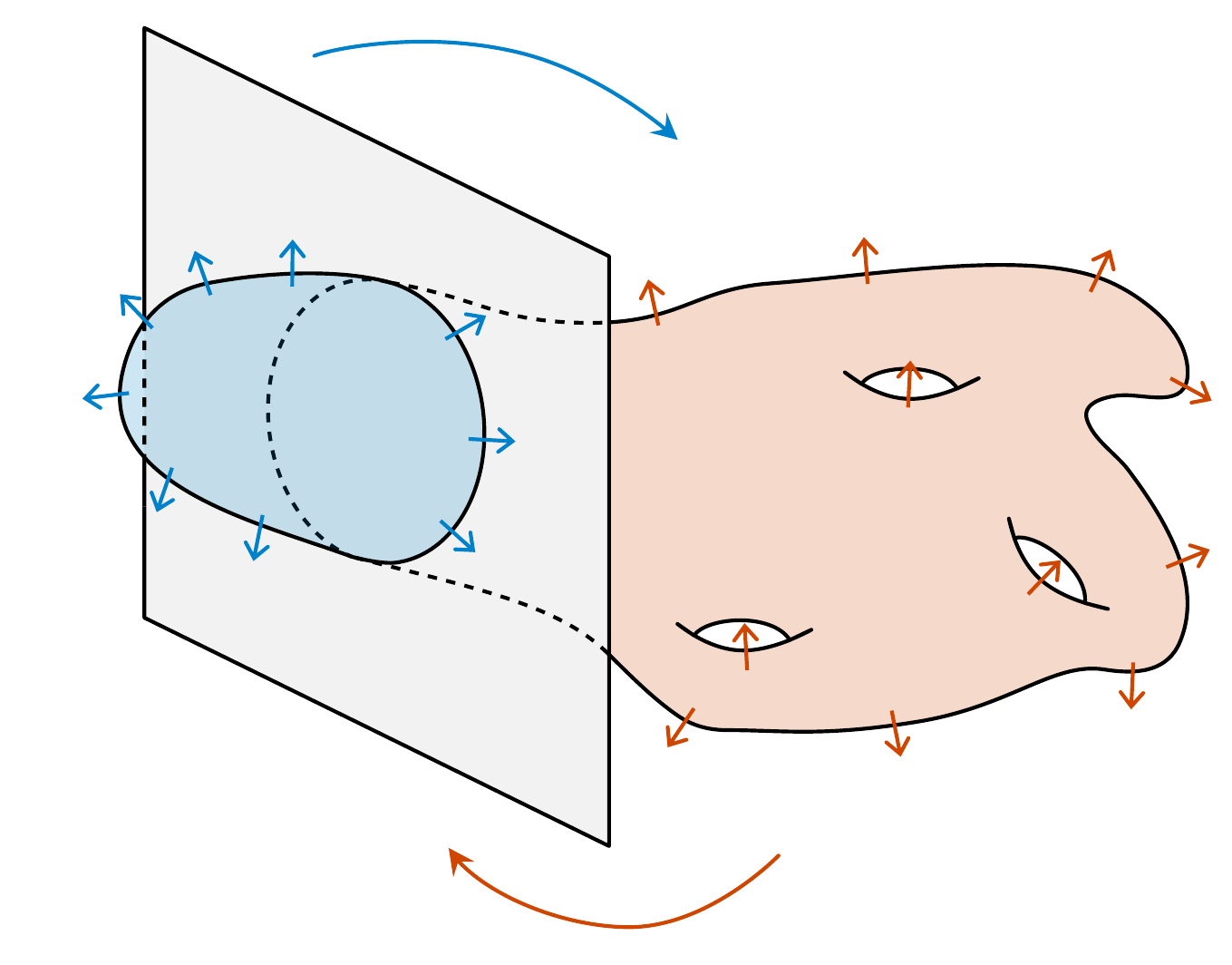}
            \put (5,37) {\color{myblue}\small $D$}
            \put (94,54) {\color{myorange}\small $\Sigma_0$}
            \put (20,13) {\small $X \times \{ 0 \}$}
            \put (60,2) {\small\color{myorange} does \textit{not} extend over $D$}
            \put (50,74) {\small\color{myblue} does \textit{not} extend over $\Sigma_0$}
        \end{overpic}
        \caption{The surface $\Sigma = D \cup \Sigma_0$ in $X \times [0,1]$ where $\nu_\Sigma$ is oriented but not trivializable with many suppressed directions.}
        \label{fig:non_triv_surface}
    \end{figure}
    \begin{proof}
        Suppose $X$ is of type~I.
        According to \cref{prop:geom_typei}, there is a closed embedded surface $\iota\colon\Sigma \hookrightarrow X$ whose normal bundle is oriented and has non-vanishing second Stiefel-Whitney class.
        Denote by $\iota_0 \colon \Sigma \hookrightarrow X \times [0,1]$ the map $s \mapsto (\iota(s), 0)$.
        We want to show that $[U_1] = 0$ in $\Fr_1(X)$.
        First, we isotope $\iota_0$ to an embedding $\iota_1 \colon \Sigma \hookrightarrow X \times [0,1]$ for which $\iota_1 \circ \pr_2$ is Morse using \cref{lem:isotope_bordism}.
        Notice that this gives an isomorphism of normal bundles $\normb_{\iota_1} \cong \normb_{\iota} \oplus \varepsilon^1$.
        Then, we can cut off a disk from $\Sigma$ whose boundary embeds into one slice of $X \times [0,1]$.
        Without loss of generality, we assume that this disk is given by the fixed disk $D$.
        That is $\partial D = U$ and the framing of $U_0$ extends over $D$.
        Because $\Sigma$ has a non-trivial normal bundle, the framing of $U_0$ cannot be extended over $\Sigma_0 = \Sigma - \mathrm{int}(D)$.
        The situation is depicted in \cref{fig:non_triv_surface}.
        Then, according to \cref{cor:unique_framing_one_boundary}, the framing of $U_1$ can be extended over $\Sigma - \mathrm{int}(D)$.
        This implies $[U_1] = 0$.

        For the converse assume that $[U_1]=0$ in $\Fr_1(X)$.
        We want to show that there is an immersed closed surface in $X$ with an oriented normal bundle whose second Stiefel-Whitney class does not vanish.
        By definition, the framing of $U_0$ extends over the given disk $D \subset X \times [0,1]$.
        Since $[U_1] = 0$, the normally framed circle $U_1$ can be extended over some surface $\Sigma_0 \subset X \times [0,1]$ with boundary $U$.
        We obtain a new closed surface $\iota_1 \colon \Sigma \hookrightarrow X \times [-1,1]$ by gluing $D$ and $\Sigma_0$ together along $U \subset X \times \{0\}$ from opposite sides.
        Visualization is provided in \cref{fig:non_triv_surface}.
        This surface has an oriented and non-trivial normal bundle of rank $n \geq 3$, and because of \cref{rmk:vb_trivial_sw_class}, the second Stiefel Whitney class $w_2(\normb_{\iota_1})$ must not vanish.
        We project the surface to $X \times \{0\}$ and homotope it to an immersed surface $\iota \colon \Sigma \looparrowright X$.
        Denote by $\iota_0 \colon \Sigma \looparrowright X \times [-1,1]$ the map $s \mapsto (\iota(s), 0)$. 
        Since $n \geq 3$, the maps $\iota_0$ and $\iota_1$ are homotopic through immersions.
        Again, we obtain $\normb_{\iota_1} \cong \normb_{\iota} \oplus \varepsilon^1$ and thus $w_2(\normb_{\iota}) = w_2(\normb_{\iota_1})$.
        We conclude that $\normb_{\iota}$ is oriented and $w_2(\normb_{\iota})$ does not vanish, and hence by \cref{prop:geom_typei}, $X$ is of type~I.
    \end{proof}

    \begin{example}\label{ex:real_projective_4k}
        Consider $X = \RP^{4k}$ for $k \geq 1$.
        Let $\RP^2 \subset \RP^{4k}$ be the standard embedding.
        A~straightforward computation with Stiefel-Whitney classes shows that $w_1(\normb_{\RP^2})$ vanishes but $w_2(\normb_{\RP^2})$ does not.%
        \footnote{
            Recall that $H^*(\RP^{n+1}; \ZZ_2) \cong \sfrac{\ZZ_2[a]}{\langle a^{n+2} \rangle}$ and $w(\RP^{n+1}) = (1+a)^{n+2}$.
        }
        Since $H_1(\RP^{4k}; o_X) \cong H^{4k - 1}(\RP^{4k}; \ZZ) = 0$, we deduce $\Fr_1(\RP^{4k}) = 0$.
    \end{example}

    \section{Manifolds of type II}
    \label{sec:typeii}
    
    Throughout this section, we assume that $X$ is of type II.
    In that case, the subgroup $\ker(h)$ is isomorphic to $\ZZ_2$.
    Hence we obtain the short exact sequence~\eqref{eq:ses:typeii} stated in the introduction, and it remains to determine the extension.

    We discuss some required background in homological algebra.
    Consider the following short exact sequence of (local) coefficients:
    \[
        \begin{tikzcd}[sep=small]
            0 \arrow[r] & o_X \arrow[r, "{\cdot \, 2}"] &[4ex] o_X \arrow[r, "{\smod 2}"] &[4ex] X \times \ZZ_2 \arrow[r] & 0 \mathrlap{.}
        \end{tikzcd}
    \]
    This gives rise to connecting homomorphisms $\beta_k\colon H_{k+1}(X;\ZZ_2) \to H_{k}(X; o_X)$ which we call \emph{twisted Bockstein homomorphisms}.
    In cohomology, we have Bockstein homomorphisms $\beta^k \colon H^{k}(X;\ZZ_2) \to H^{k+1}(X;\ZZ)$ associated with the sequence $\ZZ \to \ZZ \to \ZZ_2$.
    In a similar way to the ordinary Bockstein homomorphisms, the following diagram commutes:
    \begin{equation}\label{eq:bockstein_poincare}
        \begin{tikzcd}[sep=large]
            H^{n-k}(X;\ZZ_2)  \arrow[r, "\beta^{n-k}"] \arrow[d, "\mathrm{PD}_2"', "\cong"]    & H^{n+1-k}(X;\ZZ) \arrow[d, "\mathrm{PD}_\mathrm{tw}"', "\cong"] \arrow[r, "\smod 2"] & H^{n+1-k}(X; \ZZ_2) \arrow[d, "\mathrm{PD}_2"', "\cong"] \\
            H_{k+1}(X;\ZZ_2)  \arrow[r, "\beta_k"']                                             & H_{k}(X; o_X) \arrow[r, "\smod 2"']                                                & H_{k}(X; \ZZ_2) \mathrlap{.}
        \end{tikzcd}    
    \end{equation}
    We appeal to Munkres \cite[Lemma~69.2]{bib:munkres:elem_alg_top} for proof of the commutativity of the left square in the oriented non-twisted case.
    \begin{remark}\label{rmk:bockstein_poincare_dual}
        Note that the top row of \eqref{eq:bockstein_poincare} is precisely given by the first Steenrod square~$\sq^1$.
        Let $\iota\colon\Sigma \hookrightarrow X$ be an embedded (or immersed) surface in $X$.
        Then $\beta_1$ sends the fundamental class $\iota_*[\Sigma]_2$ to a class in $H_1(X; o_X)$ whose $\smod 2$ reduction is the Poincar\'e dual of the shriek $\iota_! w_1(\normb_\iota)$.%
        \footnote{
            The original result goes back to Thom~\cite{bib:thom:steenrod_normal} which is further generalized by Eccles and Grant~\cite{bib:eccl_grant:steenrod}.
        }
    \end{remark}

    Next, we would like to classify the extension given in~\eqref{eq:ses:typeii} and give it a geometric interpretation.
    What follows revisits the discussion in~\cite[pp.\:6--7]{bib:kir_mel_tei:coh_4manifolds} applied to the twisted case.
    Consider the free resolution
    \begin{equation}\label{eq:free_res}
        \begin{tikzcd}[sep=small]
            0 \arrow[r] & B_1 \arrow[r, "i"] & Z_1 \arrow[r] & H_1(X; o_X) \arrow[r] & 0,
        \end{tikzcd}
    \end{equation}
    where $i \colon B_1 \to Z_1$ is the inclusion of 1-boundaries into 1-cycles in the singular chain complex of the local coefficients $o_X$.
    Using this free resolution, we can identify $\Ext(H_1(X; o_X), \ZZ_2)$ with $\coker(i^*)$.
    That is, an element of $\Ext(H_1(X; o_X), \ZZ_2)$ is represented by a homomorphism $\phi \colon B_1 \to \ZZ_2$.
    Denote by $\Tor_2(H_1(X; o_X))$ the subgroup of $H_1(X; o_X)$ with all elements of order at most two.
    Then we get a homomorphism
    \begin{equation}\label{eq:ext_tor2_iso}
        \begin{split}
            \Psi \colon \Ext(H_1(X; o_X), \ZZ_2) & \to \Tor_2(H_1(X; o_X))^*, \\
            [\phi] &\mapsto \left( [z] \mapsto \phi(2z) \right),
        \end{split}
    \end{equation}
    where we denoted $\Hom(-, \ZZ_2) = -^*$.
    Notice that this map is well-defined because
    \begin{enumerate}[i)]
        \item if $[z] \in \Tor_2(H_1(X; o_X))$, we have $2z \in B_1$,
        \item for any $b \in B_1$, we have $\phi(2(z + b)) = \phi(2z)$, and
        \item if $\phi$ extends to a map $Z_1 \to \ZZ_2$, then $\phi(2z) = 0$.
    \end{enumerate}
    The next proposition is an application of~\cite[Theorem~26.5]{bib:eil_mac:group_extensions} spelled out for our case.
    \begin{prop}
        The map $\Psi \colon \Ext(H_1(X; o_X), \ZZ_2) \to \Tor_2(H_1(X; o_X))^*$ is an isomorphism.
    \end{prop}
    \begin{proof}
        Since $X$ is compact, $H_1(X; o_X)$ is a direct sum of cyclic groups.
        According to Cohen and Gluck~\cite{bib:coh_glu:stacked_bases}, there is a stacked basis for the sequence~\eqref{eq:free_res}, meaning there is a basis $\{z_i\}_{i \in I}$ of $Z_1$ so that $\{ n_i z_i \}_{i \in J}$ is a basis of $B_1$ for $J \subseteq I$ and some $n_i \in \ZZ$.
        We split $B_1 = B_1^\mathrm{even} \oplus B_1^\mathrm{odd}$ generated by the $n_i z_i$ with $n_i$ even or odd respectively.
        For injectivity, assume that $\Psi([\phi]) = 0$.
        This means that ${\phi}|_{B_1^\mathrm{even}} \equiv 0$.
        In that case, $\phi$ can be extended to a map $Z_1 \to \ZZ_2$, and therefore $[\phi]=0$ in $\coker(i^*)$ as desired.
        For surjectivity, fix a homomorphism $\varepsilon \colon \Tor_2(H_1(X; o_X)) \to \ZZ_2$.
        We define $\phi_\varepsilon \colon B_1 \to \ZZ_2$ as follows.
        On $B_1^\mathrm{even}$, we declare $\phi_\varepsilon(n_i z_i) = \varepsilon([\frac{n_i}{2} z_i])$, and otherwise we set ${\phi_\varepsilon}|_{B_1^\mathrm{odd}} \equiv 0$.
        It is now easy to verify that $\Psi([\phi_\varepsilon]) = \varepsilon$.
    \end{proof}
    \begin{lemma}\label{lem:bockstein_dual_uct}
        The Bockstein $\beta_1$ is dual to $\delta$ in the twisted universal coefficient sequence~\eqref{eq:twisted_uct}, meaning that the following diagram commutes:
        \begin{equation*}
            \begin{tikzcd}
                \Ext(H_1(X; o_X), \ZZ_2) \arrow[r, "\delta"] \arrow[d, "\cong"'] & H^2(X; \ZZ_2) \arrow[d, "\cong"] \\
                \Tor_2(H_1(X; o_X))^* \arrow[r, "\beta_1^*"] & H_2(X; \ZZ_2)^* \mathrlap{.}
            \end{tikzcd}
        \end{equation*}
    \end{lemma}
    \begin{proof}
        Under the identification $\Ext(H_1(X; o_X), \ZZ_2) \cong \coker(i^*)$, the map $$\delta\colon\Ext(H_1(X; o_X), \ZZ_2) \to H^2(X; \ZZ_2)$$ can be described as follows.
        Take a representative $\phi \colon B_1 \to \ZZ_2$ and map it to $$\phi \circ \partial \colon C_2(X; o_X) \to \ZZ_2.$$
        This is a cocycle in the cochain complex $\Hom(C_*(X; o_X), \ZZ_2)$.
        But we have a canonical chain isomorphism $\Hom(C_*(X; o_X), \ZZ_2) \cong C^*(X; \ZZ_2)$.
        Under this identification, $\phi \circ \partial$ is mapped to a cocycle representing the $\ZZ_2$-cohomology class $\delta([\phi])$.
        For the other composition, recall the construction of the Bockstein boundary homomorphism.
        Let $c \in C_2(X; \ZZ_2)$ be a cycle and $\bar{c} \in C_2(X; o_X)$ a lift of $c$, i.e.\ $\bar{c} \:\smod{2} = c$.
        Then the Bockstein boundary map is given by $\beta_1([c]) = [\frac{1}{2}\partial \bar{c}]$.
        Precomposing $\beta_1^*$ with the isomorphism in \eqref{eq:ext_tor2_iso} gives precisely $[\phi \circ \partial]$, and hence the diagram commutes.
    \end{proof}
    The discussion thus far has primarily focused on homological algebra.
    The following lemma provides a more geometric perspective.
    \begin{lemma}\label{lem:char_func_geom}
        Let $\varepsilon_X \colon \Tor_2(H_1(X; o_X)) \to \ZZ_2$ be the homomorphism associated with the extension in~\eqref{eq:ses:typeii}.
        If $[L, \framing]$ is in the preimage of $\Tor_2(H_1(X; o_X))$ under the map $h \colon \Fr_1(X) \to H_1(X; o_X)$, then
        \begin{align}\label{eq:torsion_circle_geom}
            (\varepsilon_X \circ h)([L, \framing]) = 
            \begin{cases}
                0 & \text{iff} \quad 2\,[L, \framing] = [U_0], \\
                1 & \text{iff} \quad 2\,[L, \framing] = [U_1].
            \end{cases}
        \end{align}
    \end{lemma}
    \begin{proof}
        Using the free resolution~\eqref{eq:free_res}, the extension of~\eqref{eq:ses:typeii} is represented by a map $\phi\colon B_1 \to \ZZ_2$ that is induced by $\psi \colon Z_1 \to \Fr_1(X)$ in the following commutative diagram:
        \[
            \begin{tikzcd}
                0 \arrow[r] & B_1 \arrow[r, "i"] \arrow[d, "\phi"] & Z_1 \arrow[r] \arrow[d, "\psi"] & H_1(X; o_X) \arrow[r] \arrow[d, equal] & 0 \\
                0 \arrow[r] & \ZZ_2 \arrow[r, "u"]                 & \Fr_1(X) \arrow[r, "h"]         & H_1(X; o_X) \arrow[r]                  & 0 \mathrlap{.}
            \end{tikzcd}
        \]
        For a reference see Spanier~\cite[§5.5.2]{bib:spanier:alg_top}.
        Notice that the equivalence class ${[\phi] \in \coker(i^*)}$ does not depend on the choice of $\psi$.
        Let $[L, \framing]$ be an element in the preimage of $\Tor_2(H_1(X; o_X))$ under $h$ and ${[z] = h([L, \framing])}$.
        Then $2\,[L, \framing] \in \ker(h)$, that is $2\,[L, \framing]$ is either $[U_0]$ or $[U_1]$.
        Because the diagram commutes, we obtain $(u \circ \phi)(2z) = 2\,[L, \framing]$.
        But by definition, we have $(\varepsilon_X \circ h)([L, \framing]) = \phi(2z)$, and hence~\eqref{eq:torsion_circle_geom} follows.
    \end{proof}
    \begin{remark}\label{rmk:independent_framing}
        The relation established in~\eqref{eq:torsion_circle_geom} remains invariant under any choice of initial framing~$\framing$, as the doubling of this choice has no contribution to the resulting framing, as indicated by \cref{prop:unique_framing}.
    \end{remark}
    \begin{remark}
        By \cref{prop:u1_order_two} and \cref{lem:char_func_geom}, it follows that any non-zero element $[L, \framing]$ in the preimage of $\Tor_2(H_1(X; o_X))$ under $h$ has either order two or four in $\Fr_1(X)$ depending on whether $(\varepsilon_X \circ h)([L, \framing])$ is 0 or 1.
    \end{remark}
    With our geometric perspective on the group extension, we are almost ready to approach the proof of the main theorem.
    However, before doing so, we require additional geometric tools.
	\begin{prop}\label{prop:framing_inward_stretch}
        Suppose $M$ is a closed $(m+1)$-dimensional manifold with $m \geq 1$.
        Let $\iota \colon N \hookrightarrow M$ be an embedded compact manifold of dimension $k+1$ with boundary~$\partial N$.
        Let $\tau$ be a normal framing of $\iota \colon N \hookrightarrow M$.
        Suppose $V$ is an outward-pointing vector field in $\restr{TN}{\partial N}$.
        Then we have $\left[\iota(\partial N) , \left(\restr{\tau}{\partial N}, \mathrm{d}\iota(V) \right)\right] = 0$ in $\Fr_k(M)$.
	\end{prop}
    \begin{figure}[t]
        \centering
        \def\unitlength{\textwidth}
        \begin{picture}(.5,.5)%
            \put(0,0){\includegraphics[width=.5\unitlength]{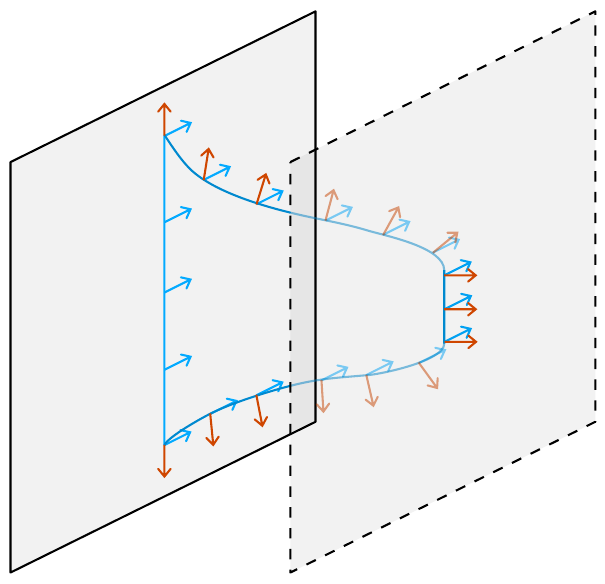}}%
            \put(0.09,.025){\small $M \times \{0\}$}%
            \put(0.325,.025){\small $M \times \{\frac{1}{2}\}$}%
            \put(0.07,.225){\color[HTML]{03AAFF}\small $\iota(N)$}%
            \put(0.175,.25){\color[HTML]{03AAFF}\small $\tau$}%
            \put(0.3,.225){\color{myblue}\small $\jmath(N)$}%
            \put(0.4,.3){\color{myorange}\small $V^{(\jmath)}$}%
        \end{picture}%
        \caption{The construction of the framed null-bordism in \cref{prop:framing_inward_stretch}.}
        \label{fig:framing_nullbordism}
    \end{figure}
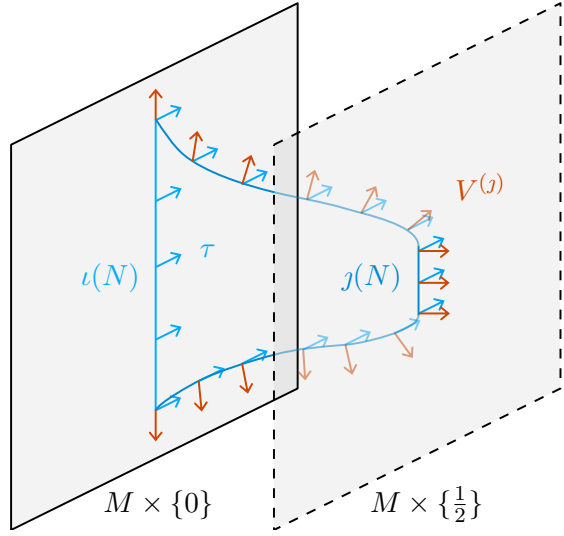
    \begin{proof}
        We take a collar neighborhood $U \subset N$ together with a diffeomorphism $\psi \colon \partial N \times [0,1) \to U$. 
        Moreover, let $f \colon N \to [0,\frac{1}{2}]$ be a smooth function such that
        \begin{enumerate}[i)]
            \item $f^{-1}(0) = \partial N$,
            \item $f(U) = [0, \frac{1}{2})$,
            \item $\restr{f}{M - U} \equiv \frac{1}{2}$, and
            \item $\left(\restr{f}{U} \circ \psi \right)(p,\,\cdot\,) \colon [0, 1) \to [0, \frac{1}{2})$ is strictly increasing for any $p \in \partial N$.
        \end{enumerate}
        Consider the embedding $\jmath \colon N \hookrightarrow M \times [0,1]$ defined via $p \mapsto \left( \iota(p), f(p) \right)$.
        We extend the vector field $V$ on $\partial N$ to a smooth vector field onto $U$ such that $\mathrm{d}f(V) < 0$ where we denote the extension with the same letter.
        Now, we define a vector field along $\jmath \colon N \hookrightarrow M \times [0,1]$ given by
        \begin{align*}
            V^{(\jmath)}(p) := \begin{cases}
                \left( (1 - 2 f(p)) \, \mathrm{d}\iota(V(p)), \; 2 f(p) \restr{\frac{\partial}{\partial{t}}}{f(p)} \right) & \text{if $p \in U$}, \\
                \left(0,\restr{\frac{\partial}{\partial{t}}}{f(p)}\right) & \text{if $p \in N - U$},
            \end{cases}
        \end{align*}
        where $t \in [0,1]$ is the coordinate in the second component.
        Because $\mathrm{d}f(V) < 0$ and $f > 0$ in $U - \partial N$, the vector field $V^{(\jmath)}$ represents a section of $\normb_{\jmath}$.
        For the normal framing $\tau = (v_1, \dots, v_{n-1})$, we denote $\tau^{(\jmath)} = \left( (v_1, 0), \dots, (v_{n-1}, 0) \right)$.
        We hence obtain a normal framing $(\tau^{(\jmath)}, V^{(\jmath)})$ of $\jmath \colon N \hookrightarrow M \times [0,1]$ that restricts to $(\restr{\tau}{\partial N}, \mathrm{d}\iota_0(V))$ on the boundary $\partial N$.
        An illustration of this construction is provided in \cref{fig:framing_nullbordism}.
    \end{proof}
    As long as we are in the stable range, meaning $2k < m$, the set $\Fr_k(M)$ is an abelian group and inverses are given by flipping the orientation of the framing.
    Hence, we obtain the following corollary:
    \begin{corollary}\label{cor:framing_in_out}
        Suppose $M$ is a closed $(m+1)$-dimensional manifold with $m \geq 1$, and suppose that $2k<m$.
        Let $\iota \colon N \hookrightarrow M$ be an embedded compact manifold of dimension $k+1$ with boundary $\partial N = B_0 \sqcup B_1$.
        Assume $\tau$ is a normal framing of $\iota \colon N \hookrightarrow M$.
        Suppose $V$ is a vector field in $\restr{TN}{\partial M}$ inward-pointing on $B_0$ and outward-pointing on $B_1$.
        Then we have in $\Fr_k(M)$:
        \begin{align*}
            \left[ \iota(B_0), \left( \restr{\tau}{B_0}, \mathrm{d}\iota\!\left(\restr{V}{B_0}\right) \right) \right] = \left[ \iota(B_1), \left( \restr{\tau}{B_1}, \mathrm{d}\iota\!\left(\restr{V}{B_1}\right) \right) \right].
        \end{align*}
    \end{corollary}
    We apply the corollary to a special situation that will play an important role in the proof of the main theorem.
    \begin{lemma}\label{lem:rank1_subbundle_bordism}
        Suppose $\lambda$ is a rank one subbundle of the trivial bundle $\varepsilon^2=S^1 \times \RR^2$.
        Let $D(\lambda) \subset \varepsilon^2$ be the disk bundle associated with $\lambda$, and $C_0$ be the zero section of~$\varepsilon^2$.
        Then $2\, [C_0, (e_1, e_2)] = [\partial D(\lambda), (v_1, v_2)]$ in $\Fr_1(\varepsilon^2)$, where $(e_1, e_2)$ is the standard frame, $v_2$ is an inward-pointing normal of $\partial D(\lambda) \subset D(\lambda)$, and $(v_1, v_2)$ is completed to a frame of $\normb_{\partial D(\lambda)}$ that is induced by the orientation of $(e_1, e_2)$.
    \end{lemma}
    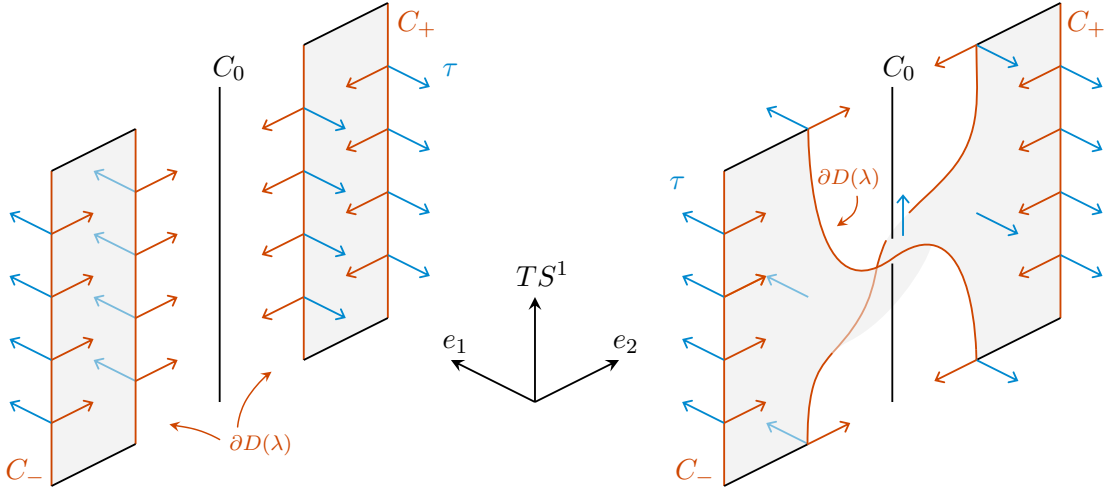
\begin{figure}[t]
        \centering
        \def\unitlength{\textwidth}
        \begin{picture}(1,.462)%
            \put(0,0){\includegraphics[width=\unitlength]{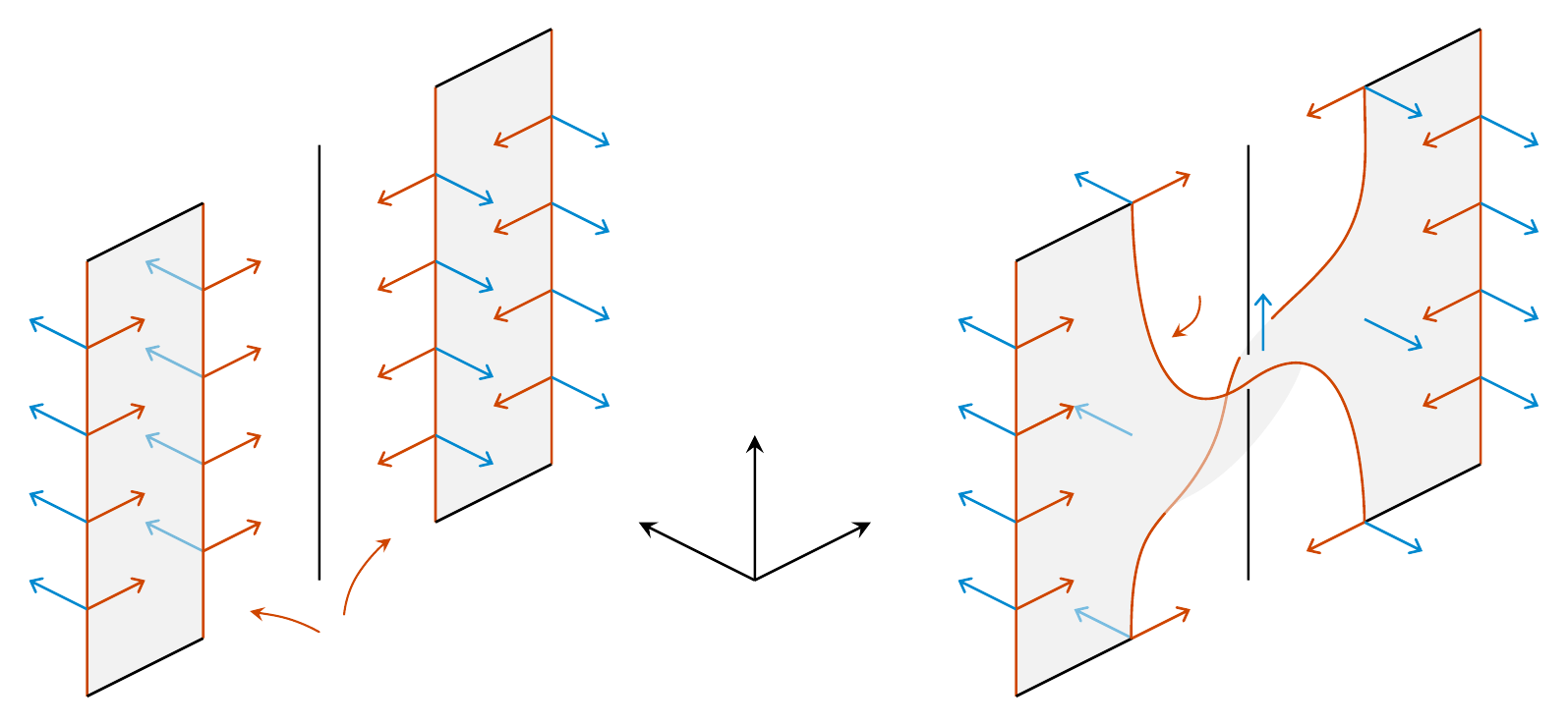}}%
            \put(0.465,.195){\small $TS^1$}%
            \put(0.55,.14){\small $e_2$}%
            \put(0.4,.14){\small $e_1$}%
            \put(0.4,.38){\color{myblue}\small $\tau$}%
            \put(0.6,.28){\color{myblue}\small $\tau$}%
            \put(0.1975,.38){\small $C_0$}%
            \put(0.2125,.05){\color{myorange}\tiny $\partial D(\lambda)$}%
            \put(0.7875,.38){\small $C_0$}%
            \put(0.7305,.285){\color{myorange}\tiny $\partial D(\lambda)$}%
            \put(0.605,.025){\color{myorange}\small $C_-$}%
            \put(0.015,.025){\color{myorange}\small $C_-$}%
            \put(0.36,.42){\color{myorange}\small $C_+$}%
            \put(0.95,.42){\color{myorange}\small $C_+$}%
        \end{picture}%
        \caption{The construction of framed bordisms in \cref{lem:rank1_subbundle_bordism}. Left the case when $\lambda$ is trivial, and right when $\lambda$ is the non-trivial bundle.}
        \label{fig:tubular_framing}
    \end{figure}
    \begin{proof}
        Let $C_\pm$ be two copies of $C_0$ displaced in positive and negative $e_2$-direction respectively.
        We endow $C_\pm$ with the normal frame $(\pm e_1, \pm e_2)$ so that $C_- \sqcup C_+$ represents $2\, [C_0, (e_1, e_2)]$.
        We must consider two cases: either $\lambda$ is trivializable or not.
        For the first case, assume without loss of generality that $\lambda = S^1 \times \{0\} \times \RR$, meaning that $e_2$ is a frame of $\lambda$.
        Then $\pm e_1$ extends trivially to a normal frame $\tau$ of $D(\lambda)$, and the statement follows from \cref{cor:framing_in_out} as depicted in the left of \cref{fig:tubular_framing}.
        Suppose that $\lambda$ is non-trivializable, that is $D(\lambda)$ is homeomorphic to the M\"obius band.
        In that case, the bordism is given by a pair of pants as depicted in the right of \cref{fig:tubular_framing}, and $\pm e_1$ over $C_\pm$ extends to a normal frame $\tau$ over the pair of pants.
        Applying \cref{cor:framing_in_out}, where $M = \varepsilon^2$, we get a normal frame on $\partial D(\lambda)$ that twists around $\partial D(\lambda)$ once.
        This frame is homotopic (through frames) to $(v_1, v_2)$ defined in the lemma, which proves the claim.
    \end{proof}
    The following proposition relates \cref{lem:char_func_geom} to normal bundles of embedded surfaces, and thus provides a deeper geometric insight.
    \begin{prop}\label{prop:surface_normal_whitney_twice}
        Assume that $X$ is of type~II, and let $\Sigma \subset X$ be a closed connected surface.
        Moreover, let $L \subset X$ be a link with an oriented normal bundle representing the class $\beta_1([\Sigma]_2) \in H_1(X; o_X)$.
        Then $w_2(\normb_\Sigma) = 0$ if and only if $2\,[L, \framing] = [U_0]$ independent of the choice of normal positive framing $\framing$ over $L$.
    \end{prop}
    \begin{proof}
        First, observe that $L$ with its orientation on $\normb_L$ is homologous in $H_1(X; o_X)$ to a closed curve $C \subset \Sigma \subset X$ with an orientation of $\normb_C$.
        This follows from naturality of the Bockstein homomorphism
        \[
            \begin{tikzcd}
                H_2(\Sigma; \ZZ_2) \arrow[r, "\beta_1"] \arrow[d, "\iota_*"'] & H_1(\Sigma; \iota^* o_X) \arrow[d, "\iota_*"] \\
                H_2(X; \ZZ_2) \arrow[r, "\beta_1"] & H_1(X; o_X),
            \end{tikzcd}
        \]
        where $\iota \colon \Sigma \hookrightarrow X$ is the inclusion map.
        Recall that the $\smod 2$ homology class of $C$ is Poincar\'e dual to $w_1(\normb_\Sigma)$ according to \cref{rmk:bockstein_poincare_dual}.
        Suppose $\normb_\Sigma$ is orientable, meaning that $[C]_2 = 0$ in $H_1(\Sigma; \ZZ_2)$.
        It follows that $C$ is bounded by a connected component of $\Sigma - C$.
        This implies that $C$ with an orientation of $\normb_C$ is null-homologous in $H_1(X; o_X)$.
        Then the statement follows automatically since $[C, \framing] \in \ker(h)$ for any normal framing $\framing$, and $w_2(\normb_\Sigma) = 0$ by \cref{cor:geom_typeii} as $X$ is of type~II.
        
        Now suppose that $\normb_\Sigma$ is non-orientable.
        As $C$ represents the Poincar\'e dual of $w_1(\normb_\Sigma)$, it follows that $\restr{\normb_\Sigma}{\Sigma - C}$ is orientable.
        Notice that since $C$ is not null-homologous, the complement $\Sigma - C$ remains connected.
        Denote with $\normb^\Sigma_C$ the normal bundle of $C$ within $\Sigma$.
        Let $N \subset \Sigma$ be the closed tubular neighborhood associated with $\normb^\Sigma_C$.
        We cut off an embedded disk $D$ of $\Sigma - N$.
        Without loss of generality, we may assume that $\partial D = U$.
        We may also assume that $U_0$ is framed by $(\tau, V)$, where $\tau$ extends to a normal framing over $D$ and $V$ is an inward-pointing vector field of $\partial D \subset D$.
        That is, $U_0$ extends over a stretched version of $D$ in $X \times [0,1]$ according to \cref{prop:framing_inward_stretch}.
        Let $\framing$ be any normal framing of $C$ in $X$.
        We want to show that $2\,[C,\framing] = [U_0]$ if and only if $w_2(\normb_\Sigma) = 0$.
        By obstruction theory, this is equivalent to showing that $2\,[C,\framing] = [U_0]$ if and only if $\normb_\Sigma$ admits $n-2$ linearly independent sections.
        Notice that we can split $\normb_C$ into line bundles as follows:
        \begin{align*}
            \normb_C \cong \restr{\normb_\Sigma}{C} \oplus \normb^\Sigma_C \cong \varepsilon^{n-2} \oplus \gamma \oplus \normb^\Sigma_C.
        \end{align*}
        Without loss of generality, we may choose the framing $\framing = (e_1, \dots, e_n)$ in such a way that $(e_{n-1}, e_n)$ is a framing of $\gamma \oplus \normb^\Sigma_C$.
        Notice that we can extend $(e_1, \dots, e_{n-2})$ over $N$ as an $(n-2)$-framing.
        By \cref{lem:rank1_subbundle_bordism}, the bordism class $2 \, [C, \framing]$ is represented by $\partial N$ with normal framing $(e_1, \dots, e_{n-2}, v_1, v_2)$, where $v_1$ is normal to $N$, and $v_2$ is the inward-pointing normal of $\partial N \subset N$.
        As discussed in \cref{prop:unique_framing}, the normal framing $\tau = (e_1, \dots, e_{n-2}, v_1)$ can be extended over $\Sigma - (N \cup D)$ to give a framing of $\restr{\normb_\Sigma}{U}$.
        If we assume $2\,[C, \framing] = [U_0]$, then the framing $\tau$ also extends over $D$.
        This implies that $(e_1, \dots, e_{n-2})$ is an $(n-2)$-framing of $\normb_\Sigma$ as desired.
        For the converse direction, assume that $(e_1, \dots, e_{n-2})$ is an $(n-2)$-framing of $\normb_\Sigma$.%
        \footnote{Such a framing predetermines the splitting $\restr{\normb_\Sigma}{C} \cong \varepsilon^{n-2} \oplus \gamma \oplus \normb^\Sigma_C$ into line bundles.}
        But since $\restr{\normb_\Sigma}{\Sigma - N}$ is oriented and $(e_1, \dots, e_{n-2})$ extends over $D$, the framing $\tau = (e_1, \dots, e_{n-2}, v_1)$ extends over $D$ as well.
        We thus obtain $2\, [C, \framing] = [U_0]$ as required.
    \end{proof}
    \begin{remark}\label{rmk:build_extension_geometrically}
        Let $C \subset X$ be a closed curve with an oriented normal bundle, representing a non-zero element of $\Tor_2(H_1(X; o_X))$.
        Since $\beta_1 \colon H_2(X; \ZZ_2) \to \Tor_2(H_1(X; o_X))$ is surjective, there exists an embedded closed surface $\iota \colon \Sigma \hookrightarrow X$ for which $C$ represents $\iota_! w_1(\normb_\Sigma)$.
        \cref{prop:surface_normal_whitney_twice} shows that $[C, \framing] \in \Fr_1(X)$ with any normal framing $\framing$ is of order two if $w_2(\normb_\Sigma) = 0$, and of order four otherwise.
    \end{remark}
    \begin{example}
        We consider $X = \RP^{n+1}$.
        Let $\RP^2 \subset \RP^{n+1}$ be the standard embedding.
        Notice that the fundamental class $[\RP^2]_2$ generates $H_2(\RP^{n+1}; \ZZ_2)$.
        The group $H_1(\RP^{n+1}; o_X)$ is of order at most two.
        This means that the twisted Bockstein $\beta_1 \colon H_2(\RP^{n+1}; \ZZ_2) \to H_1(\RP^{n+1}; o_X)$ is surjective, and $\beta_1([\RP^2]_2)$ generates $H_1(\RP^{n+1}; o_X)$.
        Notice that $w_1(\normb_{\RP^2}) = 0$ if and only if $H_1(\RP^{n+1}; o_X) = 0$.
        Consequently, if we compute $w_1(\normb_{\RP^2})$ and $w_2(\normb_{\RP^2})$, we can determine the type of $\RP^{n+1}$, and also the associated cohomotopy group $\pi^n(\RP^{n+1})$ by \cref{rmk:build_extension_geometrically}.
        The computation is similar to \cref{ex:real_projective_4k}, and the results are summarized in \cref{tab:real_projective}.
        \begin{table}[ht]
            \centering
            \begin{tabular}{r|c|c|c|c}
                $(n+1) \bmod 4$ & 0 & 1 & 2 & 3 \\
                \hline
                $w_1(\normb_{\RP^2})$ & 0 & 1 & 0 & 1 \\
                $w_2(\normb_{\RP^2})$ & 1 & 1 & 0 & 0 \\
                type & I & IIb & IIa & IIa \\
                $\pi^n(\RP^{n+1})$ & 0 & $\ZZ_4$ & $\ZZ_2$ & $\ZZ_2 \oplus \ZZ_2$
            \end{tabular}
            \caption{First two Stiefel-Whitney classes of $\normb_{\RP^2}$ inside $\RP^{n+1}$ as well as the type of $\RP^{n+1}$ and the associated cohomotopy group.}
            \label{tab:real_projective}
        \end{table}
        \vspace{-5pt}
    \end{example}
    We are now ready to prove the second part of \cref{thm:main} for type~II manifolds.
    For convenience, we restate the theorem here.
    \begin{theorem}\label{thm:typeii}
        If $X$ is of type II, then there is an extension of abelian groups
        \[
            \begin{tikzcd}[sep=small]
                0 \arrow[r] & \ZZ_2 \arrow[r] & \Fr_1(X) \arrow[r, "h"] & H_1(X; o_X) \arrow[r] & 0 \mathrlap{,}
            \end{tikzcd}
        \]
        classified by the unique element in $\Ext(H_1(X, o_X), \ZZ_2)$ that maps to ${w_1^2(X) + w_2(X)}$ via $\delta$ in the universal coefficient sequence~\eqref{eq:twisted_uct}.
    \end{theorem}
    \begin{proof}
        From the previous discussion, we have seen that the extension is uniquely determined by the functional $$\varepsilon_X \colon \Tor_2(H_1(X; o_X)) \to \ZZ_2$$ through~\eqref{eq:ext_tor2_iso}.
        Together with \cref{lem:bockstein_dual_uct}, which relates $\delta$ to the dual Bockstein, it is therefore sufficient to show that $\varepsilon_X \circ \beta_1 = w_1^2 + w_2$.
        Let $\Sigma \subset X$ be a closed connected surface representing a homology class in $H_2(X; \ZZ_2)$.
        Suppose the element ${\beta_1([\Sigma]_2) \in \Tor_2(H_1(X; o_X))}$ is represented by a normally oriented link $L$.
        Then it follows from \cref{lem:char_func_geom} and \cref{prop:surface_normal_whitney_twice} that both $(\varepsilon_X \circ \beta_1)([\Sigma]_2)$ and $\left\langle w_2(\normb_\Sigma), [\Sigma]_2 \right\rangle$ are either 0 or 1 depending on whether $2[L, \framing]$ is $[U_0]$ or $[U_1]$ for any choice of framing~$\framing$.
        This implies
        \begin{align*}
            (\varepsilon_X \circ \beta_1)([\Sigma]_2) = \left\langle w_2(\normb_\Sigma), [\Sigma]_2 \right\rangle,
        \end{align*}
        which is equal to $(w_1^2 + w_2)([\Sigma]_2)$ by~\cref{lem:func_measures_w2}.
        The statement then follows from the fact that $H_2(X; \ZZ_2)$ is generated by closed embedded surfaces.
    \end{proof}
    \begin{corollary}\label{cor:pin_split}
        The short exact sequence~\eqref{eq:ses:typeii} splits if and only if $X$ is $\Pinm$.
        Therefore, in that case, we have $\Fr_1(X) \cong H_1(X; o_X) \oplus \ZZ_2$.
    \end{corollary}
    The last \cref{cor:pin_split} generalizes the result obtained by Konstantis~\cite{bib:konstantis:counting_inv} to non-orientable manifolds.
    In particular, Konstantis examines the splitting maps and obtains a counting invariant from it.
    This is the subject of the following sections.

    \section{Manifolds of type IIa}
    \label{sec:typeiia}
    In the previous section, we have shown that~\eqref{eq:ses:typeii} splits if and only if $X$ is $\Pinm$.
    In this section, we sort out the geometric interpretation of the splitting maps.
    Throughout this section, we assume that $X$ is $\Pinm$, and we fix a $\Pinm$-structure on $X$.
    Suppose $(L_0, \framing_0)$ is a normally framed link in $X$.
    As $\normb_{L_0}$ is trivialized via $\framing_0$ and $\restr{TX}{L_0}$ carries a $\Pinm$-structure, $TL_0$ inherits a $\Pinm$-structure $\sigma_0$ via stabilization (see~\cite[Lemma~1.6]{bib:kir_tay:pin}).
    This defines a (stable) bordism class $[L_0, \sigma_0] \in \PinBor1$.
    \begin{lemma}
        The bordism class $[L_0, \sigma_0] \in \PinBor1$ as constructed above does not depend on the choice of representative in the bordism class $[L_0, \framing_0] \in \Fr_1(X)$.
    \end{lemma}
    \begin{proof}
        We assume that $(L_0, \framing_0)$ and $(L_1, \framing_1)$ are normally framed bordant through $\Sigma \subset X \times [0,1]$ with normal framing $\framing$.
        That is, $(\Sigma, \framing)$ restricts to $(L_i, \framing_i)$ when intersected with $X \times \{i\}$ for each $i=0,1$.
        Again, $T \Sigma$ inherits a $\Pinm$-structure $\sigma$ from $\restr{TX}{\Sigma}$ because $\normb_\Sigma$ is trivialized through $\framing$.
        By construction, the restriction of $\sigma$ to $X \times \{i\}$ agrees with the $\Pinm$-structure $\sigma_i$ on $TL_i$ induced by $\framing_i$.
        Therefore, $(\Sigma, \sigma)$ is a $\Pinm$ bordism from $(L_0, \sigma_0)$ to $(L_1, \sigma_1)$ which proves the claim.
    \end{proof}
    As discussed in~\cite[Theorem~2.1]{bib:kir_tay:pin}, the group $\PinBor1$ is canonically isomorphic to $\ZZ_2$ generated by the circle $S^1_{\mathrm{Lie}}$ with the Lie group framing.
    This gives rise to a well-defined group homomorphism $\kappa \colon \Fr_1(X) \to \ZZ_2$.
    Next, we show that $\kappa$ restricted to $\ker(h)$ is an isomorphism.
    \begin{prop}
        The restriction $\restr{\kappa}{\ker(h)} \colon \ker(h) \to \ZZ_2$ is an isomorphism.
    \end{prop}
    \begin{proof}
        The statement follows from the proof of $\PinBor1 \cong \ZZ_2$ in~\cite[Theorem~2.1]{bib:kir_tay:pin}.
        Since the generator of $\PinBor1$ is given by the circle $S^1_{\mathrm{Lie}}$ with the Lie group framing, its induced $\Pinm$-structure does not extend over a disk.
        But $\ker(h)$ is generated by $[U_1]$ as shown in \cref{lem:ker_h_generator}, and $\kappa$ maps $[U_1]$ to the generator $[S^1_{\mathrm{Lie}}]$ in $\PinBor1$.
    \end{proof}
    This proposition shows that $\kappa$ is a splitting map for the short exact sequence~\eqref{eq:ses:typeii}.
    For the rest of this section, we show how exactly those splitting maps and $\Pinm$-structures are tied together.
    Let $\PinStr(X)$ be the set of $\Pinm$-structures up to equivalence.
    Furthermore, we denote
    \begin{align*}
        \Split(X) := \{ \text{splitting maps $\kappa \colon \Fr_1(X) \to \ZZ_2$ of~\eqref{eq:ses:typeii}} \}.
    \end{align*}
    The construction in the preceding paragraphs gives rise to a map $\PinStr(X) \to \Split(X)$.
    We show that this map satisfies an equivariant-like property with certain group actions.

    First note that there is a simply transitive group action by ${\Hom(H_1(X; o_X), \allowbreak \ZZ_2)}$ upon $\Split(X)$.
    The group action is described as follows.
    Take a splitting map $\kappa \colon \Fr_1(X) \to \ZZ_2$ and a homomorphism $\gamma \colon H_1(X; o_X) \to \ZZ_2$.
    If $(L, \framing)$ is a normally framed link in $X$, we define the action of $\gamma$ on $\kappa$ by
    \begin{align*}
        (\kappa \cdot \gamma)\left( [L, \framing] \right) := \kappa([L, \framing]) + (\gamma \circ h)([L, \framing]).
    \end{align*}
    It is not hard to check that this action is transitive and free.
    \begin{lemma}\label{lem:action_isomorphism}
        The mod\:2 reduction $r \colon H_1(X; o_X) \to H_1(X; \ZZ_2)$ induces an isomorphism
        \begin{align*}
            \Hom(H_1(X; o_X), \ZZ_2) \xrightarrow{\cong} \sfrac{H^1(X; \ZZ_2)}{\left<w_1(X)\right>}.
        \end{align*}
    \end{lemma}
    \begin{proof}
        Consider the twisted universal coefficient theorem in cohomology (see \cref{thm:twisted_uct}):
        \begin{equation*}
            \begin{tikzcd}[sep=small]
                0 \arrow[r] & \Ext(H_0(X; o_X), \ZZ_2) \arrow[r] & H^1(X; \ZZ_2) \arrow[r, "r^*"] & \Hom(H_1(X; o_X), \ZZ_2) \arrow[r] & 0 \mathrlap{.}
            \end{tikzcd}
        \end{equation*}
        Since $X$ is connected, we have
        \begin{align*}
            H_0(X; o_X) \cong    \begin{cases}
                                        \ZZ & \text{if $X$ is orientable,} \\
                                        \ZZ_2 & \text{if $X$ is non-orientable.}
                                    \end{cases}
        \end{align*}
        This implies that $\Ext(H_0(X; o_X), \ZZ_2)$ is either 0 or $\ZZ_2$ depending on whether $X$ is orientable or not.
        Thus $\ker(r^*)$ is either 0 or $\ZZ_2$ depending on whether $X$ is orientable or not.
        Let $w_1 \colon H_1(X; \ZZ_2) \to \ZZ_2$ denote the homomorphism induced by $w_1(X)$.
        Recall that elements of $H_1(X; o_X)$ are represented by normally oriented links $L$.
        This means that $\restr{TX}{L}$ must be orientable for all such $L$.
        Then, by the definition of $w_1(X)$, we deduce $(w_1 \circ r)([L]) = 0$ for all such $L$.
        Therefore ${\ker(r^*) = \left<w_1(X)\right>}$, and as $r^*$ is an epimorphism, we conclude the statement.
    \end{proof}
    The last ingredient allows us to draw a connection to the group action by $H^1(X; \ZZ_2)$ upon $\PinStr(X)$.
    In particular, \cref{thm:splitting_pin} is an immediate consequence of the following equivariant property:
    \begin{theorem}\label{thm:pin_split_equivariant}
        The following diagram of obvious maps commutes:
        \[
            \begin{tikzcd}[sep=large]
                \PinStr(X) \times H^1(X; \ZZ_2) \arrow[r, "\mathrm{act}"] \arrow[d]                                 & \PinStr(X) \arrow[d] \\
                \Split(X) \times \sfrac{H^1(X; \ZZ_2)}{\left<w_1(X)\right>} \arrow[r, "\mathrm{act}"]     & \Split(X) \mathrlap{.}
            \end{tikzcd}
        \]
    \end{theorem}
    \begin{proof}
        We fix a $\Pinm$-structure on $X$.
        Let $\kappa\colon \Fr_1(X) \to \ZZ_2$ be the corresponding induced splitting map.
        Suppose $\gamma \in \sfrac{H^1(X; \ZZ_2)}{\left<w_1(X)\right>}$, and let $\bar{\gamma} \in H^1(X; \ZZ_2)$ be a lift of $\gamma$.
        Moreover assume $(C, \framing)$ is a normally framed circle in $X$.
        When we identify $H^1(X; \ZZ_2) \cong \Hom(H_1(X; \ZZ_2), \ZZ_2)$, the action of $\bar{\gamma}$ on the given $\Pinm$-structure can be described as follows:
        \begin{align}\label{eq:change_of_pin}
            \bar{\gamma}([C]_2) =     \begin{cases}
                                        0 & \text{iff $\bar{\gamma}$ does \emph{not} change the $\Pinm$-structure on $\restr{TX}{C}$}, \\
                                        1 & \text{iff $\bar{\gamma}$ changes the $\Pinm$-structure on $\restr{TX}{C}$}.
                                    \end{cases}
        \end{align}
        Because $\restr{TX}{C}$ is orientable, we have $w_1([C]_2) = 0$.
        This implies that $\bar{\gamma}([C]_2)$ does not depend on the choice of lift.
        With the identification in \cref{lem:action_isomorphism}, we conclude
        \begin{align}
            (\kappa \cdot \gamma)([C, \framing])     &= \kappa([C, \framing]) + (\gamma \circ h)([C, \framing]) \nonumber\\
                                                    &= \kappa([C, \framing]) + \bar{\gamma}([C]_2). \label{eq:change_of_split}
        \end{align}
        The statement follows from substituting~\eqref{eq:change_of_pin} into~\eqref{eq:change_of_split}.
    \end{proof}

    \section{Application to vector bundles}
    \label{sec:vector_bundles}

    By obstruction theory, a choice of $\Spin$-structure over an oriented vector bundle of rank 3 or higher is equivalent to a choice of framing over the 1-skeleton that extends over the 2-skeleton.
    This leads to the following lemma:
    \begin{lemma}\label{lem:spin_unique_framing}
        Let $B$ be a 1-dimensional cell complex and $E \to B$ be an oriented spin vector bundle of rank $n \geq 3$.
        Then $E$ is isomorphic to the trivial bundle and a $\Spin$-structure uniquely determines a framing of~$E$.
    \end{lemma}
    A geometric proof of the above statement can be found in~\cite[Lemma~3.1]{bib:konstantis:counting_inv}.
    For the rest of this section, we assume that $n \geq 4$.
    Suppose $E \to X$ is an oriented spin vector bundle of rank $n$, and we fix a $\Spin$-structure over~$E$.
    Let $s \colon X \to E$ be a global section, and let $0_E \colon X \hookrightarrow E$ be the zero section.
    We may assume that $s$ intersects $0_E$ transversally, and we denote $s^{-1}(0_E) := L$.
    As the intersection is transverse, the differential $\mathrm{d}s \colon TX \to TE$ defines a canonical isomorphism of bundles $\normb_L \cong \restr{E}{L}$.
    The bundle $\restr{E}{L}$ inherits a $\Spin$-structure from $E$.
    By \cref{lem:spin_unique_framing} this induces a unique framing of $\restr{E}{L}$ and therefore a normal framing $\framing$ of $\normb_L$.
    This defines a framed bordism class $[L, \framing] \in \Fr_1(X)$.
    The class $h([L, \framing]) \in H_1(X; o_X)$ is the (twisted) Poincar\'e dual to the Euler class~$e(E) \in H^n(X; \ZZ)$.
    In other words, the class $[L, \framing]$ is a refinement of the Euler class~$e(E)$.
    We will show that this class is a \textit{full} obstruction to finding a non-vanishing section on~$E$.

    \begin{lemma}\label{lem:section_well_defined}
        The class $[L, \framing]$ does not depend on the choice of section $s$ up to homotopy.
    \end{lemma}
    \begin{proof}
        Suppose that $s' \colon X \to E$ is a section transverse to $0_E$ that is homotopic (through sections) to $s$.
        That is, there is a homotopy $\bar{s} \colon X \times [0,1] \to E$ with $\bar{s}(\,\cdot\,, 0) = s$ and $\bar{s}(\,\cdot\,, 1) = s'$.
        Consider the projection $\pr_1 \colon X \times [0,1] \to X$ and denote $\bar{E} = \pr_1^* E$.
        Then $\bar{s}$ is a section of $\bar{E}$ and without loss of generality, we may assume that $\bar{s}$ is transverse to $0_{\bar{E}}$.
        Denote the zero locus of $s'$ by $L' := (s')^{-1}(0_E)$ and the induced normal framing by $\varphi'$.
        Without loss of generality, we assume that the submanifold $\Sigma := \bar{s}^{-1}(0_{\bar{E}}) \subset X \times [0,1]$ has no connected components without boundary.
        Then the boundary of $\Sigma$ is precisely $L$ and $L'$ and $\Sigma \simeq \text{1-dim.\ cell complex}$.
        The restricted bundle ${\bar{E}}|_{\Sigma}$ inherits a $\Spin$-structure from $E$, and by \cref{lem:spin_unique_framing}, this gives a unique framing of ${\bar{E}}|_{\Sigma}$.
        This framing induces a normal framing over $\Sigma$ so that $(L, \framing)$ is normally framed to $(L', \framing')$.
    \end{proof}

    \begin{lemma}\label{lem:independent_triv}
        The class $[L, \framing]$ does not depend on the choice of the $\Spin$-structure over~$E$ if $w_n(E) = 0$.
    \end{lemma}
    \begin{proof}
        To begin, let us observe that $\mathrm{PD}_2([L]_2) = w_n(E)$.
        Consequently, we can deduce that $w_n(E) = 0$ if and only if $[L]_2 = 0$.
        A change of $\Spin$-structure corresponds to a homomorphism $\pi_1(X) \to \ZZ_2$, i.e.\ an element in $H^1(X; \ZZ_2)$.
        That homomorphism takes the homotopy class of a circle $C$ and ``twists'' the framing over $C$ by an element in $\pi_1(\SO(n)) \cong \ZZ_2$.
        However, given that $[L]_2 = 0$, any such change does not impact the framing over~$L$.
    \end{proof}

    Let $W$ be a connected closed $(n+k)$-dimensional manifold with connected boundary~$\partial W$, where $k \geq 1$.
    Consider a smooth map $f \colon W \to \RR^n$ for which zero is a regular value and $f^{-1}(0) \cap \partial W = \varnothing$.
    Then $M := f^{-1}(0)$ is a closed submanifold of $\mathrm{int}(W)$.
    The manifold $M$ comes equipped with an induced normal framing $\framing$.
    Note that $(M, \framing)$ defines an element $[M, \framing]$ in the \emph{relative} framed bordism group $\Fr_k(W, \partial W)$.%
    \footnote{
        Here, $\Fr_k(W, \partial W)$ is the set of normally framed closed $k$-dimensional submanifolds of $\mathrm{int}(W)$ up to framed bordism.
    }
    Now, we~set
    \begin{align*}
        g \colon \partial W \to S^{n-1} \quad \text{where} \quad g(x) := \frac{f(x)}{|f(x)|}.
    \end{align*}
    \begin{lemma}\label{lem:extension_of_unit}
        The map $g$ extends over $W$ if and only if $[M, \framing]=0$ in $\Fr_k(W, \partial W)$.
    \end{lemma}
    \begin{proof}
        To begin the proof, we observe that the map $f \colon W \to \RR^n$, viewed as a map into the 1-point compactification $\RR^n \subset \RR^n \cup \{\infty\} = S^n$, can be homotoped to a map $f_{\mathrm{rel}} \colon (W, \partial W) \to (S^n, \infty)$ by radially extending it to infinity in a collar of $\partial W$.
        The Pontryagin-Thom construction gives an isomorphism $$\Fr_k(W, \partial W) \cong [(W, \partial W), (S^n, \infty)].$$
        This isomorphism associates the bordism class $[M, \framing]$ with the homotopy class $[f_\mathrm{rel}]$.
        When we apply the Puppe sequence to the inclusions
        \begin{align*}
            \iota &\colon \partial W \hookrightarrow W,\\
            \jmath &\colon (W, \varnothing) \hookrightarrow (W, \partial W),
        \end{align*}
        we obtain a (long) exact sequence in relative cohomotopy sets:
        \begin{equation}\label{eq:puppe_sequence}
            \begin{tikzcd}[sep=small]
                \left[ \susp W, S^n \right] \arrow[r, "\iota^*"] & \left[ \susp \partial W, S^n \right] \arrow[r, "\delta"] & \left[ (W, \partial W), (S^n, \infty) \right] \arrow[r, "\jmath^*"] & \left[ W, S^n \right],
            \end{tikzcd}
        \end{equation}
        where $\mathrm{S}$ denotes the unreduced suspension functor.
        Since we assumed $n \geq 4$, we are in the stable range for the suspension and obtain $$[\susp W, S^n] \cong [W, S^{n-1}] \quad \text{and} \quad [\susp \partial W, S^n] \cong [\partial W, S^{n-1}].$$
        To understand the boundary map $\delta$ in~\eqref{eq:puppe_sequence}, we consider the suspended map $\mathrm{S}g \colon \mathrm{S}\partial W \to S^n$.
        Using the inclusion of a collar $\partial W \times [0,1] \hookrightarrow W$, we can extend $\mathrm{S}g$ onto $W$ by assigning it a constant value outside the collar, which produces a map $f'_{\mathrm{rel}} \colon (W, \partial W) \to (S^n, \infty)$.
        By construction, the map $f'_{\mathrm{rel}}$ is homotopic to $f_\mathrm{rel}$ relative to $\partial W$.
        Consequently, we have $\delta[g] = [f_{\mathrm{rel}}]$.
        The statement now follows from the exactness of~\eqref{eq:puppe_sequence}.
    \end{proof}
    We are ready to prove the main theorem of this section that generalizes the invariant introduced by Konstantis~\cite{bib:konstantis:counting_inv}.
    \begin{theorem}\label{thm:nonvanishing_section}
        Let $E \to X$ be an oriented vector bundle of rank $n \geq 4$ with a $\Spin$-structure.
        Moreover, denote by $(L, \framing)$ the zero locus of a section transverse to the zero section together with the normal framing induced by the $\Spin$-structure.
        Then $E$ admits a non-vanishing section if and only if $[L, \framing] = 0$ in $\Fr_1(X)$.
    \end{theorem}
    \begin{proof}
        Suppose $E$ admits a non-vanishing section $s'$.
        Then the zero locus of $s'$ is the empty set, and because of \cref{lem:section_well_defined}, we get $[L, \varphi] = 0$.

        For the converse, assume that $[L, \framing] = 0$ in $\Fr_1(X)$.
        Let $\Sigma \subset X \times [0,1]$ be the associated framed null-bordism.
        Notice that $\Sigma$ may not be connected, i.e.\ $\Sigma = \Sigma_1 \sqcup \dots \sqcup \Sigma_m$.
        Without loss of generality, we assume that each connected component $\Sigma_k$ has a non-empty boundary.
        We isotope each $\Sigma_k$ relative to~$\partial \Sigma_k$ to a submanifold $\iota_k \colon \Sigma_k \hookrightarrow X$.
        There, we attach a closed collar, i.e.\ we extend $\iota_k$ to an embedding
        \begin{align*}
            \Sigma_k \cup_{\partial \Sigma_k} \big( \partial \Sigma_k \times [0,1] \big) \hookrightarrow X,
        \end{align*}
        and consider the tubular neighborhood of this embedding.
        The corresponding disk bundle gives us a compact manifold with corners.
        We can smoothen the corners to produce a compact manifold $W_k$ with connected and smooth boundary $\partial W_k$.
        By construction, we have $[\partial \Sigma_k, \framing_k] = 0$ in $\Fr_1(W_k, \partial W_k)$, where $\framing_k$ is the restriction of $\framing$ to $\partial \Sigma_k$.
        As $W_k$ is homotopy equivalent to $\Sigma_k$, we have $\restr{E}{W_k} \cong W_k \times \RR^n$.
        If we restrict $s$ to $W_k$, we get an induced map $f_k \colon W_k \to \RR^n$ whose zero locus is precisely~$\partial \Sigma_k$.
        We can now apply \cref{lem:extension_of_unit} to deform $s$ into a section that has no zeroes in each~$W_k$.
    \end{proof}

    \begin{remark}
        The proof above can be interpreted as a version of the Whitney trick.
        To see this, we apply the same trick if the rank of $E$ is equal to the dimension of $X$, namely $n+1$.
        In that case, the null-bordism is given by a bunch of arcs that isotope to arcs in $X$.
        We concatenate these arcs with their corresponding lifts along the section $s$ and obtain closed curves in the total space $E$ that bound contractible disks.
        These disks can now be used to ``push'' $s$ from $0_E$ in a neighborhood of each arc.
    \end{remark}

    Now, we would like to relate \cref{thm:nonvanishing_section} to type I and type II manifolds.
    If $X$ is a manifold of type~I, the map $h \colon \Fr_1(X) \to H_1(X; o_X)$ is an isomorphism.
    Therefore, \cref{thm:vb_typei} is an immediate consequence of \cref{thm:nonvanishing_section}.
    To obtain a well-defined invariant in the case of type IIa manifolds, we require the next lemma.
    \begin{lemma}
        Suppose that $X$ is of type IIa, and choose a $\Pinm$-structure over $X$.
        We denote by $\kappa \colon \Fr_1(X) \to \ZZ_2$ the induced splitting map.
        Moreover, let $E$ and $(L, \framing)$ as in \cref{thm:nonvanishing_section}.
        If $w_n(E) = 0$, then $\kappa([L, \framing])$ does not depend on the chosen $\Pinm$-structure.
    \end{lemma}
    \begin{proof}
        The proof is similar to the proof of \cref{lem:independent_triv} and follows from~\cref{thm:pin_split_equivariant}.
        If $w_n(E)=0$, we obtain from~\eqref{eq:change_of_split} that $\kappa([L, \framing])$ remains unchanged for any choice of $\Pinm$-structure.
    \end{proof}
    Together with \cref{lem:independent_triv}, this allows us to define \emph{the degree of $E$}
    \begin{align*}
        \kappa(E) := \kappa([L, \framing]),
    \end{align*}
    provided that $w_n(E)=0$.
    If $X$ is oriented, this quantity is precisely the counting invariant introduced by Konstantis~\cite{bib:konstantis:counting_inv}.
    In the same spirit, we obtain \cref{thm:vb_typeiia}, which is now a corollary of \cref{thm:nonvanishing_section}, bearing in mind that $w_n(E)$ is the $\smod 2$ reduction of~$e(E)$.

    \printbibliography

\end{document}